\documentclass[12pt,times]{article}

\usepackage{amssymb,amsmath,exscale,times,amsthm}

\usepackage{graphicx}
\usepackage{dsfont,stmaryrd}
\usepackage[applemac]{inputenc}
\usepackage[T1]{fontenc}

\usepackage{scalerel}

\usepackage{color}
\definecolor{marin}{rgb}   {0.,   0.3,   0.7} 
\definecolor{rouge}{rgb}   {0.8,   0.,   0.} 
\definecolor{sepia}{rgb}   {0.8,   0.5,   0.} 
\usepackage[colorlinks,citecolor=marin,linkcolor=rouge,
            bookmarksopen,
            bookmarksnumbered
           ]{hyperref}

\newtheorem{lemma}{Lemma}[section]

\newtheorem{theorem}[lemma]{Theorem}
\newtheorem{proposition}[lemma]{Proposition}
\newtheorem{corollary}[lemma]{Corollary}
\newtheorem{remark}[lemma]{Remark}
\newtheorem{example}[lemma]{Example}

\newtheorem{notation}[lemma]{Notation}
\newtheorem{definition}[lemma]{Definition}
\newtheorem{conclusion}[lemma]{Conclusion}
\numberwithin{equation}{section}
\newcommand{\QED}{\mbox{}\hfill \raisebox{-0.2pt}{\rule{5.6pt}{6pt}\rule{0pt}{0pt}} 
          \medskip\par}

\newcommand{\ad}{\mathrm{ad}}
\newcommand{\Ac}{\mathcal{A}}

\newcommand{\dd}{\mathrm{d}}

\newcommand{\Lc}{\mathcal{L}}

\newcommand{\N}{\mathbb{N}}

\newcommand{\R}{\mathbb{R}}
\newcommand{\Fc}{\mathcal{F}}

\newcommand{\C}{\mathbb{C}}

\newcommand{\T}{\mathbb{T}}
\newcommand{\Z}{\mathbb{Z}}

  \newcommand{\e}{{\rm e}}

\newcommand{\Norm}[2]{\|#1\|\left.\vphantom{T_{j_0}^0}\!\!\right._{#2}} 
\newcommand{\DNorm}[2]{\llbracket#1\rrbracket\left.\vphantom{T_{j_0}^0}\!\!\right._{#2}}         
             
\title{Discrete pseudo-differential operators and applications to numerical schemes}        
\author{Erwan Faou{\small$\,^1$} and Beno\^it Gr\'ebert{\small$\,^2$}\\[4ex]
{\small$\,^1$}\,
\small 
INRIA Rennes, Univ Rennes \& Institut de Recherche Math\'ematiques de Rennes,\\[-1ex] 
\small CNRS UMR 6625 Rennes \&
ENS Rennes, France.\\[-1ex]
\small Campus Beaulieu F-35042 Rennes Cedex, France.\\[-1ex]
\it \small email: \tt Erwan.Faou@inria.fr\\[2ex]
{\small$\,^2$}\, \small Laboratoire de Math\'ematiques Jean Leray,
Universit\'e de Nantes,\\[-1ex]
\small 2, rue de la Houssini\`ere
F-44322 Nantes cedex 3, France. \\[-1ex]
\small\it email: \tt benoit.grebert@univ-nantes.fr
}       
\begin{document}
\maketitle
\begin{abstract}
We define a class of discrete operators acting on infinite, finite or periodic sequences  mimicking the standard properties of pseudo-differential operators. In particular we can define the notion of order and regularity, and we recover the fundamental property that the commutator of two discrete operators gains one order of regularity. We show that standard differential operators acting on periodic functions, finite difference operators and fully discrete pseudo-spectral methods fall into this class of discrete pseudo-differential operators. As examples of practical applications, we revisit standard error estimates for the convergence of splitting methods, obtaining in some Hamiltonian cases no loss of derivative in the error estimates, in particular for discretizations of general waves and/or water-waves equations. Moreover, we give an example of preconditioner constructions inspired by normal form analysis to deal with the similar question for more general cases. 

\end{abstract}


\tableofcontents

\section{Introduction} 

The usefulness of pseudodifferential calculus is no longer in question \cite{H\"or87,Tay81} and if forms one of the corner stone of the analysis of Partial Differential Equations (PDEs). Indeed, even for simple differential operators, the notion of inverse, flow, approximations and any functional calculus in term of the operator can be expressed in term of pseudodifferential operators. 

The undeniable power of this calculus resides, in particular, in a property  of the commutators:  if two pseudodifferential operators $A$ and $B$ are of order $r_1$ and $r_2$ respectively, then the order of the commutator $[A,B]$ is $r_1+r_2-1$ and not $r_1+r_2$ like the product $AB$. This property plays a central role in the consistency of the definition of pseudodifferential calculus, and is fundamental for obtaining estimates allowing to define the existence of solutions or to analyze their long time behavior. To cite a few examples, it has in particular been used in normal forms theory (see for instance \cite{BGMR}) but also in the definition solution with low regularity \cite{DL89} or on the analysis of scattering effects \cite{GV85} which often rely on commutator estimates. 

Unfortunately, when we represent these operators in a Hilbert basis of the space where they act (typically the Fourier basis), we obtain matrices and we lose this property: the regularity of the operators leads to a certain decrease of the coefficients of the matrices which represent them, but this notion of decrease is usually not quantified in such a way that the "miracle of the commutators" is preserved. We thus lose an important property of the differential operators when we model them by a numerical scheme. 
Despite this difficulty, commutator estimates plays a fundamental role in numerical analysis, in particular for splitting methods, see \cite{JL,Fao12,CCFM17},  or Magnus expansions and  Baker-Cambell-Hausdorf formula \cite{HLW06}. But in general classical estimates can only be obtained in specific explicit cases and yield in general to loss of derivatives (although striking recent progresses has been done in this direction, see for instance \cite{ORS21}). 

\medskip

Otis Chodosh overcame this difficulty in his thesis (see \cite{Cho1, Cho2}) by characterizing the space of infinite matrices representing pseudo-differential operators. This space encodes a notion of order which leads to a property of commutators similar to the one enjoyed by the space of pseudo-differential operators. It is this space of infinite matrices that we take up here (see Definition \ref{cho}) and that we call space of pseudo-differential matrices.  In passing we re-demonstrate the "miracle of the commutators" in a direct way. We then show that the standard pseudo-differential operators used in the world of PDEs are represented by matrices of this type (see Section \ref{diffOp}). 
One of the main result of this paper is then to extend Chodosh's class to a context of {\em periodic matrices} which allows us to show that standard discretized models of PDEs including finite difference methods, spectral or pseudo-spectral methods, are in fact represented by pseudo-differential periodic matrices (see section \ref{per}) and satisfy the wished commutator estimates. Of course to be useful, the important point is that these estimates are {\em uniform} with respect to the discretization parameters, to be consistent with the continuous limit.  With these results in hand, we expect to revisit and extend many results of analysis and numerical analysis within this new framework. It include particular  space discretized equations (or more generally lattice dynamics like FPUT or Discrete nonlinear Schr\"odinger equations), time discretization of PDEs (semi-discrete or fully discrete) or long time behavior of numerical schemes. 
 
 \medskip

To highlight all the advantages we can get from this class of matrices, we give several examples where the discrete pseudo-differential algebra yields new results. In section \ref{error}, we revisit and amplify (see Theorem \ref{thm-local}) a result of Jahnke and Lubich (see \cite{JL} ) on error estimates in splitting schemes. Their result was based on an assumption of gain of regularity of the commutator between the two part of the splitting, which is immediately satisfied if the operators concerned are in our class. Thus while Jahnke and Lubich verify this assumption only in the case of the Schr\"odinger equation (or more generally PDEs where the commutators can be calculated explicitely, {\em i.e.} involve functions and classical differential operators) we can ensure with our technics that it is automatically satisfied for any reasonable PDE and their space discretizations.  Moreover the result extends to operators of arbitrary orders without restrictions. As a consequence of this abstract analysis, we can characterize examples of situations where convergence of splitting schemes is guaranteed without any loss of regularity. This means that we can estimate the error committed by the scheme in the same regularity as that imposed on the solution we are estimating. 
\\

We notice that this latter situation occurs for splitting methods of high order as long as the modeled operator is of order strictly smaller than one (see Remark \ref{rk2}). We develop an example of application from the water waves models in section \ref{water} where the linear water wave operator is a pseudo-differential operator of order $\frac12$. Moreover, as a consequence of our analysis, the same result holds for fully discrete models obtained by spectral or finite difference approximations.   \\
Nevertheless this constraint to start from an operator of order strictly smaller than one is very restrictive and clearly not satisfied in many interesting cases. As a second example of application of our discrete pseudo-differential calculs, in section \ref{precon}, we show that by using technics inspired by  normal forms theory (see in particular \cite{BGMR}) we can circumvent this constraint by changing the unknown. This change of variable acts as a {\em preconditioner} for the splitting scheme concerned and it is described by a discrete pseudo-differential operator that we can estimate. We develop this technique in the framework of the Schr\"odinger equation (see section \ref{precon}) but we could describe the method in an abstract way in the formalism of the normal forms, as done for instance in \cite{BGMR}. In this paper, we focus on the emblematic example of the Schr\"odinger equation as proof of concept, and we propose in Proposition \ref{yesmyfriend} a pre- and post- processed Lie-splitting scheme of order one to approximate the solutions of the  Schr\"odinger equation \eqref{LS0} without loss of derivative. 

Note that the use of our class of operators and commutator estimates could certainly be very useful for time dependent problems and the analysis of Magnus integrators, both from the convergence point of view or on the obtention of long time estimates. These analysis will be the subject of further studies. 

In Section \ref{sobgrr} we also consider as application of our analysis the problem of Sobolev norm growth for linear problem with time dependent potential. We give a fairly general result that applies both to continuous and space discretized models. This also shows that the assumptions made for splitting propagators are fulfilled in many situations. 

\medskip

Notations: In the following, $x \lesssim  y$ means that $x \leq C y$ for a constant independent of $x$ and $y$, and 
$x \lesssim_\alpha y$ means that $C$ depends on $\alpha$. 

\subsection*{Acknowledgement} 
During the preparation of this work the two authors benefited from the support of the Centre Henri Lebesgue ANR-11-LABX- 0020-01 and B.G. was supported 
 by ANR -15-CE40-0001-02  ``BEKAM''   of the Agence Nationale de la Recherche. Furthermore B.G. thanks INRIA and particularly the MINGuS project for hosting him for a semester.

\section{A class of pseudo differential matrices}
We consider the  space of summable squares of complex or real numbers $\ell^2(\Z^d)$ indexed by $\Z^d$, $d \in \N$  a positive integer. Typically, a sequence of this space represents (an approximation of) the Fourier coefficients of a function defined on a periodic torus $\T^d := (\R / 2\pi \Z)^d$, see Section \ref{diffOp} for examples.\\ 
 For $s\in\R$ we define the discrete Sobolev space, $h^s$, by
$$h^s:=\{(x_k)_{k\in\Z^d}\in \C^{\Z^d}\mid \sum_{k\in\Z^d}(1+|k|)^{2s}|x_k|^2 <+\infty\}$$
where we will use the norm $|k| = |k_1| + \cdots +|k_d|$ for elements $k = (k_1,\ldots,k_d)$ in $\Z^d$ 
and we equip this Hilbert space with its natural norm: $\Norm{x}{s}=\big( \sum_{k\in\Z^d}(1+|k|)^{2s}|x_k|^2\big)^{\frac12}$. We note that the dual of $h^s$ is $h^{-s}$. 
An infinite matrix, $A: \ \Z^d\times\Z^d\mapsto \C$ is identified with the collection of its complex elements $A := \{A(m,n)\}_{(m,n) \in \Z^d \times \Z^d}$. 

For an infinite matrices $A: \ \Z^d\times\Z^d\mapsto \C$ and $j \in \{1,\ldots,d\}$, we denote $A^{j,+}$ and $A^{j,-}$ the infinite matrices defined by
$$A^{j,+}(m,n)=A(m+e_j,n+e_j)\quad \text{and}\quad A^{j,-}(m,n)=A(m-e_j,n-e_j), $$
where $e_j$ denotes the element of $\Z^d$ with components $(e_j)_{n} = \delta_{jn}$, the Kronecker symbol, for $n = 1, \ldots, d$. 
Then we define the following operators  on matrices: $\Delta_j^+ A=A^{j,+}-A$ and $\Delta_j^{-}A=A^{j,-}-A$ and for $\alpha \in \Z$ we define 
$$
\Delta_j^\alpha = \left|
\begin{array}{l}
(\Delta_j^+)^{\alpha} \quad \mbox{if} \quad \alpha \geq 0, \\[2ex]
(\Delta_j^-)^{|\alpha|} \quad \mbox{if} \quad \alpha \leq 0. 
\end{array}
\right.
$$
By convention $\Delta_j^0={\rm Id}$. For $\alpha=(\alpha_j)_{j=1,\cdots,d}\in\Z^d$ we define the finite difference operator $\Delta^\alpha$ acting on the set of infinite matrices $A: \ \Z^d\times\Z^d\mapsto \C$  by 
$$\Delta^\alpha=\Delta_1^{ \alpha_1}\cdots\Delta_d^{\alpha_d}.$$
Following the definition introduced by Chodosh (see \cite{Cho1, Cho2}), we define: 
\begin{definition}\label{cho}
Let $r\in\R$ we define the class $\Ac_r$ of pseudo differential matrices of order $r$ by:\\ $A\in\Ac_r$ if for all $\alpha\in\Z^d$  and all $N\in\N$ there exists $C_{N,\alpha}\geq 0$ such that
$$\big| (\Delta^\alpha A)(m,n)\big|\leq C_{N,\alpha}(1+|m|+|n|)^{r-|\alpha|}(1+|m-n|)^{-N},\quad \forall m,n\in\Z^d$$
where $|m|=|m_1|+\cdots+|m_d|$ and $|\alpha| = |\alpha_1| + \cdots + |\alpha_d|$. 
\end{definition}
We denote  $$\Ac=\cup_{r\in\R}\Ac_r$$
which form a graded algebra, as the next lemma will show.\\
For $r\in\R$, $\Ac_r$ is a Frechet space when equipped with the family of semi-norms
$$\Norm{A}{\alpha,N,r}:= \sup_{m,n\in \Z^d}\frac{\big| (\Delta^\alpha A)(m,n)\big| (1+|m-n|)^{N}}{(1+|m|+|n|)^{r-|\alpha|}}.$$
Note that as $\Delta^\alpha A$ is linear in $A$, we immediately have the estimate 
$$
\Norm{A + B}{\alpha,N,r} \leq \Norm{A}{\alpha,N,r} + \Norm{B}{\alpha,N,r} 
$$
and in particular, if $A$ is of order $r$ and $B$ of order $r'$, then $A+ B$ is of order $\max(r , r')$. 

\subsection{Associated operators}

\begin{lemma}\label{op}
Let $r\in\R$. An infinite matrix $A\in\Ac_r$ naturally defines a continuous operator, still denoted by $A$, from $h^s$ to $h^{s-r}$ for any $s\in\R$ via the formula
\begin{equation}\label{Ax}
(Ax)_m=\sum_{n\in\Z^d}A(m,n)x_n, \quad m\in\Z^d.\end{equation}
 Furthermore we have
 $$\Norm{A}{\Lc(h^s,h^{s-r})}\leq C \Norm{A}{0,|s|+|r|+d+1,r}$$
 for some constant $C$ depending only on $d$, $r$ and $s$. 
\end{lemma}
\proof The proof is quite standard, for the sake of completeness, we include it.\\
Let $x\in h^s$ and $A\in\Ac_r$. We denote $\tilde x$ the element of $\ell^2(\Z^d)\equiv h^0$ defined by $\tilde x_k=(1+|k|)^{s}x_k$, $k\in\Z^d$.  We have
\begin{align*}
\Norm{ Ax}{s-r}^2&=\sum_n(1+|n|)^{2(s-r)}\Big|\sum_k A(n,k)x_k\Big|^2\\
&\leq \Norm{A}{0,N,r}^2\sum_n \Big(\sum_k\frac{(1+|n|+|k|)^r(1+|n|)^{s-r}}{(1+|n-k|)^N(1+|k|)^{s}} |\tilde x_k|\Big)^2
\end{align*}
Now we use the fact that 
$$
(1+|n|+|k|)^r\lesssim_r (1+|n|)^r(1+|n-k|)^{|r|} 
$$
for all $r \in \R$. Indeed, for $r \geq 0$, it is a consequence of the triangle inequality 
$$
1+|n|+|k| \leq 1 + 2|n| + |n-k| \leq 2(1 + |n|) (1 + |n-k|) 
$$
and for $r < 0$, the inequality if equivalent to 
$$
(1+|n|+|k|)^{|r|}\gtrsim_r (1+|n|)^{|r|}(1+|n-k|)^{-|r|} 
$$
which is implied by 
$$
(1 + |n|) \leq (1+|n|+|k|)(1+|n-k|)
$$
which is true by the triangle inequality again. 
Similarly, we have $(1+|n|)^s\lesssim_r  (1+|n-k|)^{|s|}(1+|k|)^s$, and from these inequalities, we deduce that 
 $$
 \Norm{ Ax}{s-r}^2 \lesssim_{r,s} \Norm{A}{0,N,r}^2\sum_n \Big(\sum_k\frac{   |\tilde x_k|}{(1+|n-k|)^{N-|s|-|r|}}\Big)^2. 
 $$
 
 Now by using Young inequality  for convolutions (see Lemma \ref{lemyoung} in appendix), we have for $N-|s|-|r|\geq d+1$
$$\sum_n \Big(\sum_k\frac{   |\tilde x_k|}{(1+|n-k|)^{N-|s|-|r|}}\Big)^2=\Big\| \Big(\frac{1}{(1+|k|)^{N-|s|-|r|}}\Big)_{k\in\Z^d}*\tilde x  \Big\|_{\ell^2}^2\lesssim \Norm{ \tilde x}{\ell^2}^2$$
as soon as $N-|s|-|r|\geq d+1$. For the smallest choice of $N$, we thus have 
\begin{align*}
\Norm{Ax}{s-r}^2&\leq C \Norm{A}{0,N,r}^2\Norm{ x}{s}^2
\end{align*}
for some constant $C$ depending only on $d$, $r$ and $s$. 
\endproof

\subsection{Product and commutator}
We generalize the matrix product as follows: let $A,B: \ \Z^d\times\Z^d\mapsto \C$ we define, when the series converge, the product $AB$ by the formula
$$(AB)(m,n)=\sum_{k\in\Z^d}A(m,k)B(k,n)=\sum_{j=1}^d\sum_{k_j\in\Z}A(m,\sum_{i=1}^dk_i e_i)B(\sum_{i=1}^dk_i e_i,n).$$
We denote by $[A,B]:=AB-BA$ the commutator of $A$ and $B$.

We now introduce a convenient notion of interval of multi-indices: for $\alpha,\beta\in\Z^d$ we say that $\beta\in[0,\alpha]$ if for all $j=1,\cdots,d$, $0\leq \beta_j\leq \alpha_j$ or $\alpha_j\leq\beta_j\leq0$ and we say that $\beta\in[0,\alpha)$ if for all $j=1,\cdots,d$, $0\leq \beta_j\leq \alpha_j-1$ or $\alpha_j+1\leq\beta_j\leq0$.
The main result of this section is
\begin{proposition}\label{Prop-Struct} Let $r_1,r_2\in\R$.
\begin{itemize} 
\item[(i)] The matrix product is a continuous map from $\Ac_{r_1}\times\Ac_{r_2}\to \Ac_{r_1+r_2}$. More precisely, for all $\alpha\in\Z^d$  and all $N\in\N$ there exists $C(\alpha,N,r_1,r_2)>0$ such that
\begin{equation}
\label{estprod}\Norm{ AB}{\alpha,N,r_1+r_2}\leq C(\alpha,N,r_1,r_2)\Big(\sum_{\beta\in [0,\alpha]} \Norm{ A} {\beta,N+1+|r_2|,r_1}\Big)\Big(\sum_{\beta\in [0,\alpha]} \Norm{ B}{\beta,N+1+|r_1|,r_2}\Big).
\end{equation}
\item[(ii)] The commutator gains one order : the map $\Ac_{r_1}\times\Ac_{r_2}\in(A,B)\mapsto [A,B]\in \Ac_{r_1+r_2-1}$ is continuous. More precisely, for all $\alpha\in\Z^d$  and all $N\in\N$ there exists $C(\alpha,N,r_1,r_2)>0$ such that
\begin{equation}\label{-1}\Norm{ [A,B]}{\alpha,N,r_1+r_2-1}\leq C(\alpha,N,r_1,r_2)\Big(\sum_{\substack{\beta,\gamma\in [0,\alpha]\\ M_1,M_2\leq 2N+|r_1|+|r_2|}} \Norm{ A}{\beta,M_1,r_1} \Norm{ B}{\gamma,M_2,r_2}\Big).\end{equation}
\end{itemize}
\end{proposition}
This result could be deduced from the original works of Chodosh \cite{Cho1, Cho2} by using the correspondence between $\Ac_r$ and the class of pseudo-differential operators of order $r$ on $\T^d$. Here we give a direct proof which gives the specified estimates  \eqref{estprod} and \eqref{-1} and  will have the advantage to be directly carried over to numerical discretizations by finite difference, spectral or pseudo-spectral methods (see next Sections). 
%

The direct proof of  Proposition \ref{Prop-Struct} requires some technical lemmas, the reader who is not interested in the technical aspects concerning the semi-norms introduced in Definition \ref{cho} can go directly to the section \ref{diffOp}.\\
We begin with some algebraic calculus:
\begin{lemma}\label{alg}
Let $A,B$ two infinite matrices $A,B: \ \Z^d\times\Z^d\mapsto \C$ and $j=1,\cdots,d$
\begin{itemize}
\item[(i)] $\Delta_j^+(AB)=\Delta_j^+ A\ B^{j,+}+A\ \Delta_j^+ B$,
\item[(ii)] $\Delta^{-}_j(AB)=\Delta^{-}_j A\ B^{j,-}+A\ \Delta^{-}_j B$,
\item[(iii)] $\Delta_j^+ [A,B]=[\Delta_j^+ A,B]+[A,\Delta_j^+  B]+[\Delta_j^+ A,\Delta_j^+ B]$,
\item[(iv)] $\Delta^{-}_j [A,B]=[\Delta^{-}_jA,B]+[A,\Delta^{-}_j  B]+[\Delta^{-}_j A,\Delta^{-}_j B]$.
\end{itemize}

\end{lemma}
\proof It is just a calculus, for instance for (i):
\begin{align*}\Delta_j^+(AB)(m,n)&=\sum_k A(m+e_j,k)B(k,n+e_j)-\sum_k A(m,k)B(k,n)\\
&=\sum_k A(m+e_j,k+e_j)B(k+e_j,n+e_j)-\sum_k A(m,k)B(k,n)\\
&=\sum_k \Delta_j^+ A(m,k) B^{j,+}(k,n)+ \sum_k A(m,k)\Big(B(k+e_j,n+e_j) - B(k,n)\Big)\\
&=(\Delta_j^+ A\ B^{j,+})(m,n)+(A\ \Delta_j^+ B)(m,n).
\end{align*}

And for (iii):
\begin{align*}
\Delta_j^+ [A,B] &= \Delta_j^+ (AB) - \Delta_j^+ (BA) \\
&= \Delta_j^+ A\ B^{j,+}+A\ \Delta_j^+ B - \Delta_j^+ B\ A^{j,+}-B\ \Delta_j^+ A\\
&= \Delta_j^+ A\ \Delta_j^+ B  + \Delta_j^+ A \ B +A\ \Delta_j^+ B - \Delta_j^+ B\ \Delta_j^+ A -\Delta_j^+ B\  A  -B\ \Delta_j^+ A. 
\end{align*}
\endproof

\begin{lemma}\label{prod}
Let $A\in\Ac_{r_1}$ and $B\in \Ac_{r_2}$ for $r_1$ and $r_2$ in $\R$. Then
$$\Norm{AB}{0,N,r_1+r_2}\lesssim_{N,r_1,r_2} \Norm{A}{0,N+d+|r_2|,r_1}\Norm{B}{0,N+d+|r_1|,r_2}.$$
\end{lemma}
\proof
By definition we have for all $m,n\in\Z^d$
\begin{multline*}|(AB)(m,n)|\\ \leq \Norm{A}{0,N+d+|r_2|,r_1}\Norm{B}{0,N+d+|r_1|,r_2}\sum_k \frac{(1+|m|+|k|)^{r_1}(1+|n|+|k|)^{r_2}}{(1+|m-k|)^{N+|r_2|+d}(1+|n-k|)^{N+d+|r_1|}}.
\end{multline*}
So it suffices to prove that
\begin{equation}\label{oui}
\sum_k \frac{(1+|m|+|k|)^{r_1}(1+|n|+|k|)^{r_2}}{(1+|m-k|)^{N+|r_2|+d}(1+|n-k|)^{N+d+|r_1|}}\lesssim_{N,r_1,r_2}
\frac{(1+|m|+|n|)^{r_1+r_2}}{(1+|m-n|)^{N}}.
\end{equation}
To begin with, we deal with the denominators and the sum. Using that $(1+|m-k|)(1+|n-k|)\geq(1+|m-n|)$, we note that 
$$
\sum_{k\in\Z^d}\frac1{(1+|m-k|)^{N+d}(1+|n-k|)^{N+d}} \leq  (1+|m-n|)^{-N} 
\sum_{k\in\Z^d}\frac1{(1+|m-k|)^{d}(1+|n-k|)^{d}}. 
$$
But we have 
$$
\sum_{k\in\Z^d}\frac1{(1+|m-k|)^{d}(1+|n-k|)^{d}}  = \sum_{k\in\Z^d}\frac1{(1+|k|)^{d}(1+|n - m -k|)^{d}}
$$
We decompose the last sum into two parts according to  $|n-m - k| \leq \frac{|k|}{2}$ or not. As $( 1 + |k|)( 1 + |m - n -k |) \geq (1 + |m-n|)$ we obtain the bound
$$
\sum_{|n-m - k| \leq \frac{|k|}{2}}\frac1{(1+|m-n|)^{d}}+ \sum_{|n-m - k| > \frac{|k|}{2}}\frac{2^d}{(1+|k|)^{2d}}\lesssim_d 1
$$
for all $n$ and $m$, 
as in the first sum, we have $|k| \leq 2 |n-m|$ and thus the number of term is $\mathcal{O}(|n-m|^d)$ up to constants depending on $d$. 
We thus see that to prove \eqref{oui}, it remains to prove that for all $k\in\Z^d$
$$(1+|m|+|k|)^{r_1}\lesssim_{r_1}(1+|m|+|n|)^{r_1}(1+|n-k|)^{|r_1|}$$
and similarly
$$(1+|n|+|k|)^{r_2}\lesssim_{r_2}(1+|m|+|n|)^{r_2}(1+|m-k|)^{|r_2|}.$$
Let us prove the first one. We consider two cases
\begin{itemize}
\item either $r_1\geq0$ then   we use that, either $|k|\leq2|n|$ or $|n-k|\geq|k|/2$ which leads to in both cases to $(1+|m|+|k|)\leq 2(1+|m|+|n|)(1+|n-k|)$;
\item either $r_1\leq 0$ then we use that $(1+|m|+|n|)\leq (1+|m|+|k|)(1+|n-k|)$ as in the proof of Lemma \ref{op}. 
\end{itemize} 
\endproof
\begin{lemma}\label{com}
Let $A\in\Ac_{r_1}$ and $B\in \Ac_{r_2}$ then
\begin{equation}\label{comu}\Norm{[A,B]}{0,K,r_1+r_2-1}\lesssim_{K,r_1,r_2}\sum_{\substack{|\alpha| + |\beta| = 1\\ N,M\leq 2(K+|r_1|+|r_2|+d+1)}} \Norm{A}{\alpha,N,r_1}\Norm{B}{\beta,M,r_2}.\end{equation}
\end{lemma}
\proof
Let $m,n\in\Z^d$, we have
\begin{align*}[A,B](m,n)&=\sum_{k\in\Z^d}A(m,k)B(k,n)-B(m,k)A(k,n)\\
&=\sum_{\ell\in\Z^d}A(m,m+\ell)B(m+\ell,n)-B(m,n-\ell)A(n-\ell,n).
\end{align*}
Furthermore by using telescopic summation, we have 
$$|B(m,n-\ell)-B(m+\ell,n)|\leq \sum_{\substack{j\in[0,\ell) \\\alpha_i = \mathrm{sign}(\ell_i)}} (|\Delta_1^{\alpha_1} B|+\cdots+|\Delta_d^{\alpha_d} B|)(m+j,n+j-\ell) $$
and similarly
$$|A(n-\ell,n)-A(m,m+\ell)|\leq \sum_{\substack{j\in[0,m-n+\ell) \\\alpha_i = \mathrm{sign}( m - n + \ell)}}(|\Delta_1^{\alpha_1} A|+\cdots+|\Delta_d^{\alpha_d} A|)(n-\ell+j,n+j).$$
Therefore we get
\begin{equation}\label{AB}
|[A,B](m,n)|\leq \sum_{\ell\in\Z^d}\Sigma_{1,\ell}+\Sigma_{2,\ell}\end{equation}
where
$$\Sigma_{1,\ell}=|A(n-\ell,n)|\sum_{\substack{j\in[0,\ell) \\\alpha_i = \mathrm{sign}(\ell_i)}}(|\Delta_1^{\alpha_1} B|+\cdots+|\Delta_d^{\alpha_d} B|)(m+j,n+j-\ell) $$
and
$$\Sigma_{2,\ell}= |B(m+\ell,n)|\sum_{\substack{j\in[0,m-n+\ell) \\\alpha_i = \mathrm{sign}( m - n + \ell)}}(|\Delta_1^{\alpha_1} A|+\cdots+|\Delta_d^{\alpha_d} A|)(n-\ell+j,n+j).$$
So it remains to estimate $\Sigma_{1,\ell}$ and $\Sigma_{2,\ell}$. We begin with $\Sigma_{1,\ell}$:
$$
\Sigma_{1,\ell}\leq d\sum_{|\beta| = 1}\Norm{A}{0,2N,r_1}\Norm{B}{\beta,2N,r_2}\sum_{j\in[0,\ell)}\frac{(1+|m+j|+|n+j-\ell|)^{r_2-1}(1+|n-\ell|+|n|)^{r_1}}{(1+|m-n+\ell|)^{2N}(1+|\ell|)^{2N}}.
$$
We now want to choose $N$ such  that
\begin{equation}
\label{si}\sum_{j\in[0,\ell)}\frac{(1+|m+j|+|n+j-\ell|)^{r_2-1}(1+|n-\ell|+|n|)^{r_1}}{(1+|m-n+\ell|)^{2N}(1+|\ell|)^{2N}}\lesssim_{K,r_1,r_2} \frac{(1+|m|+|n|)^{r_1+r_2-1}}{(1+|m-n|)^K(1+|\ell|)^{d+1}}
\end{equation}
in such a way that
\begin{equation}
\label{si1}\sum_{\ell\in\Z^d}\Sigma_{1,\ell}\lesssim_{K,r_1,r_2,d}\sum_{|\beta| = 1}\Norm{A}{0,2N,r_1}\Norm{B}{\beta,2N,r_2} \frac{(1+|m|+|n|)^{r_1+r_2-1}}{(1+|m-n|)^K}.
\end{equation}
by summing in $\ell$ the series $\sum_{\ell} \frac{1}{(1+|\ell|)^{d+1}} < +\infty$.
To prove \eqref{si}, we  
first note that
$$(1+|m-n+\ell|)^2(1+|\ell|)^2\geq (1+|m-n|)(1+|\ell|)$$
and thus
\begin{align*}
\sum_{j\in[0,\ell)}&\frac{(1+|m+j|+|n+j-\ell|)^{r_2-1}(1+|n-\ell|+|n|)^{r_1}}{(1+|m-n+\ell|)^{2N}(1+|\ell|)^{2N}}\\ &\leq \sum_{j\in[0,\ell)}\frac{(1+|m+j|+|n+j-\ell|)^{r_2-1}(1+|n-\ell|+|n|)^{r_1}}{(1+|m-n|)^{N}(1+|\ell|)^{N}}.
\end{align*}
On the other hand we have for $j\in[0,\ell)$
$$
(1+|m+j|+|n+j-\ell|) \leq (1 + |m| + |n| + |j| + |j-\ell|) \leq 2 (1 + |m| + |n|)(1 + |\ell|)
$$
and 
\begin{align*}
(1+|m|+|n|) &\leq 1 + |m+j| + |n+j-\ell| + |j| + |j-\ell| \\
&\leq 2(1 + |m+j| + |n+j-\ell|) (1 + |\ell|).
\end{align*}
This shows that 
$$\frac{(1+|m|+|n|)}{2(1+|m-n|)(1+|\ell|)}\leq (1+|m+j|+|n+j-\ell|)\leq 2(1+|m|+|n|)(1+|\ell|)$$
and similarly
$$ \frac{(1+|m|+|n|)}{2(1+|m-n|)(1 + |\ell|)}\leq(1+|n-\ell|+|n|)\leq 2(1+|m|+|n|)(1+|\ell|).$$
By using these inequalities according to the sign of $r_2 -1$ and $r_1$, 
this leads to \eqref{si} with $N=K+|r_1|+|r_2|+d+1$. The estimate 
 \begin{equation}\label{si2}\sum_{\ell\in\Z^d}\Sigma_{2,\ell}\lesssim_{K,r_1,r_2}\sum_{|\alpha| = 1}\Norm{A}{\alpha,2N,r_1}\Norm{B}{0,2N,r_2} \frac{(1+|m|+|n|)^{r_1+r_2-1}}{(1+|m-n|)^K}.\end{equation}
 is obtained in the same way for the same choice of $N$ and thus, combining \eqref{AB} with \eqref{si1} and \eqref{si2} we get \eqref{comu}.
\endproof

{\it Proof of Proposition \ref{Prop-Struct} }
Assertion {\em (i)}  is a consequence of Lemma \ref{prod} combined with Lemma \ref{alg}, assertions {\em (i)} and {\em (ii)} by noticing that $\Norm{B^{j,+}}{\alpha,N,r} \lesssim_{\alpha,N,r} \Norm{B}{\alpha,N,r}$ and a similar relation for $B^{j,-}$. 
Assertion {\em (ii)}  is a consequence of Lemma \ref{alg} assertions {\em (iii)} and {\em (iv)}, and Lemma \ref{com}. 
\endproof
\subsection{Representation of differential operators}\label{diffOp}
Let $\T^d = (\R / (2 \pi \Z))^d$ be the standard $d$-dimensional torus. 
A complex function $u: \T^d \to \C$ in $L^2(\T^d)$ is identified with its Fourier coefficients 
$$
\hat u(k) = \frac{1}{(2\pi)^d} \int_{\T^d} e^{- i k \cdot x} u(x) \dd x, \quad k \in \Z^d
$$
and we have the correspondence, for $s\geq0$,
$$
\hat u\equiv(\hat u(k))_{k \in \Z^d} \in h^s \qquad \Longleftrightarrow \quad \partial_x^\alpha u \in L^2(\T^d), \quad \alpha \in \N^d,\quad  |\alpha| \leq s
$$
where for a multiindex $\alpha = (\alpha_1,\ldots,\alpha_d) \in \N^d$, $\partial_x^\alpha = \partial_{x_1}^{\alpha_1}
\cdots \partial_{x_d}^{\alpha_d}$.

For any function $\Phi:\R^d\to \C$, we  define $\Phi(-i \partial_x)$ by the formula
\newcommand\reallywidehat[1]{\arraycolsep=0pt\relax%
\begin{array}{c}
\stretchto{
  \scaleto{
    \scalerel*[\widthof{\ensuremath{#1}}]{\kern-.5pt\bigwedge\kern-.5pt}
    {\rule[-\textheight/2]{1ex}{\textheight}} 
  }{\textheight} %
}{0.5ex}\\           
#1\\                 
\rule{-1ex}{0ex}
\end{array}
}
$$
\reallywidehat{(\Phi( -i \partial_x) u)} (k) := \Phi(k) \hat u(k) .
$$
We notice that $\Phi(-i \partial_x)$ is  a linear operator in $\ell^2$ and we denote by $A_\Phi$ the corresponding diagonal matrix with components 
\begin{equation}
\label{Aphi}
A_\Phi(m,n) = \Phi(m) \delta_{mn}. 
\end{equation}
For a given function $V: \T^d \to \C$, we associate the operator 
$$
u(x) \mapsto V(x) u(x) 
$$
which, in Fourier, corresponds to the convolution operator
$$
\widehat{(Vu)}(k) = \sum_{\ell \in \Z^d} \widehat V(k-\ell) \hat u(\ell).
$$
We associate to the function $V$ an infinite matrix $B_V$ with components
\begin{equation}
\label{BV}
B_V(m,n) = \widehat{V}(m-n), \quad m,n\in\Z^d
\end{equation}
in such a way that we have
$$\widehat{(Vu)}=B_V  \hat u.$$
\begin{lemma}
\label{phietv}
Let $\Phi: \R^d \to \C$ and $V: \T^d \to \C$ two functions and $A_\Phi$ and $B_V$ the matrices defined above. 
\begin{itemize}
\item[(i)]
If $\Phi$ is $\mathcal{C}^\infty$ and if there exists $r \in \R$ such that 
for all $\alpha \in \N^d$ and $x \in \R^d$,  \(|\partial_x^\alpha \Phi (x)| \lesssim_{r,\alpha} \langle x\rangle^{r - |\alpha|}\), then $A_\Phi \in \Ac_r$. 
\item[(ii)] If $V$ is $\mathcal{C}^\infty$ then $B_V \in \Ac_0$. 
\end{itemize}
\end{lemma}
\begin{proof}
The operator $A_\Phi$ is diagonal and we have to prove that 
$$
|(\Delta^\alpha A_\Phi)(m,m)| \leq C_{\alpha} (1 + |m|)^{r - |\alpha|}
$$
for all $\alpha \in \Z^d$. In such an expression, $\Delta_j^+$ and $\Delta_j^-$ are finite difference operators and acting only on $\Phi$. By using classically Taylor estimates, we have for instance 
$$
(\Delta_j^+ A_\Phi)(m,m) = \Phi(m + e_j) - \Phi(m) = \partial_{x_j}\Phi(m) + \int_0^1  ( 1 - t)  \partial_{x_j}^2 \Phi(m + t e_j) \dd t
$$
which yields, by iterating this formula, estimates of the form 
\begin{equation}
\label{estDA}
(\Delta^\alpha A_\Phi)(m,m) \lesssim_\alpha |\partial_x^{|\alpha|} \Phi(m)|+ \sup_{y - m \in  [-|\alpha|,|\alpha|]^d} |\partial_{x}^{|\alpha| + 1} \Phi(y)|, 
\end{equation}
showing {\em (i)} under the assumption on $\Phi$. 
\medskip 

\noindent To prove {\em (ii)}, we first notice that $\Delta_j^+ B_V = \Delta_j^- B_V = 0$ as $B_V(m,n)$ depends only on $m-n$. Hence we only need to prove that for all $N$, 
$$
\big|B_V(m,n)\big| = \big|\hat V(m-n)\big| \lesssim_N (1+|m-n|)^{-N},\quad \forall m,n\in\Z^d
$$
which holds true for $V \in \mathcal{C}^\infty (\T^d)$. 
\end{proof}
As a consequence of this result and Proposition \eqref{Prop-Struct}, we obtain: 
\begin{corollary}
\label{cor28}
Let $\Phi_i$, $i = 1, \ldots, P$ functions satisfying condition (i) of the previous Lemma, for orders $r_i \in \R$, and $V_i$ some smooth functions. Then 
\begin{equation}
\label{AAA}
A = \prod_{i = 1}^P A_{\Phi_i} B_{V_i} \in \Ac_{r}, \quad \mbox{with} \quad r = r_1 + \ldots + r_P. 
\end{equation}
\end{corollary}

With this result in hand, we see that all the standard pseudo-differential operators leads, in the Fourier side, to matrices belonging to a class $\Ac_r$ for a well chosen $r$\footnote{Actually this should also be recover from \cite{Cho1}.}. For example: 
\begin{itemize}
\item Fourier multipliers defined as polynomials of $D := - i \nabla_x$ which is the multiplication by $k \in \Z^d$ in Fourier. This includes the standard Laplace operator and linear KdV operator for instance. 
\item Transport operators of order one, of the form $u \mapsto \mathrm{div} ( \rho(x) u)$ or $u \mapsto X(x)\cdot \nabla u$ for some smooth function $\rho$ or smooth vector field $X$, 
\item Order two operators of the form $u \mapsto \mathrm{div} (a(x) \nabla u)$ for some smooth function $a$, 
\item All the pseudo-differential arising in fluid mechanics, for example the water wave operator 
\begin{equation}
\label{ww}\Omega^2:=\frac{1}{\sqrt{\mu}}|D| \tanh( \sqrt{\mu} |D|) 
\end{equation}
where $\mu$ is a small parameter and $|D|$ the Fourier multiplier $(|k_1|,|k_2|)$ in 2D. This operator encodes the pseudo-differential Dirichlet-to-Neumann operator arising in water wave theory. Note that the operator $\Omega$ is of order $r = \frac{1}{2}$. 
\end{itemize}
We can also extend the definition of $\Ac_r$ to operators acting on vector fields with components in $h^s$: 
\begin{definition}
\label{defsys}
Let $p \geq 1$ be a given integer. Let $\boldsymbol{r} = (r_{ij})_{1 \leq i,j \leq p}$ be a matrix of integers. 
We say that $\boldsymbol{A} \in \mathcal{A}_{\boldsymbol{b}}$ if 
$\boldsymbol{A}= (A_{ij})_{1 \leq i,j  \leq p}$ is $p$ matrices of elements $A_{ij} \in \mathcal{A}_{r_{ij}}$ for all $i,j \in \{1,\ldots,p\}$. 
\end{definition}
We can then extend the notion of product and commutators and norms from the components $A_{ij}$ to the system operator $\boldsymbol{A}$. In particular, the component of the product $(\boldsymbol{A} \boldsymbol{B})_{i,j} = \sum_{k = 1}^p A_{ik} B_{jk}$ is of order $\max_{k} (r_{ik} + r_{kj})$ and similar formula for the commutators. 

 For example for any vector field $\boldsymbol{u} = (u_1,u_2,u_3): \T^d \to \R^3$ of zero average on $\T^d$, denoting by $\hat u_i(k)$, $i = 1,2,3$ the Fourier transform of its components, the Leray projection of this field onto the field of divergence free vector field is given by 
$$
P\boldsymbol{u} = \boldsymbol{u} - \nabla \Delta^{-1}( \nabla \cdot \boldsymbol{u})
$$
and can be written 
$$
(\widehat{P\boldsymbol{u}})_i(k) = \hat u_i(k) - \sum_{j = 1}^3 \frac{k_{i}k_j}{|k|^2} \hat u_j(k), \quad k = (k_1,k_2,k_3) \in \Z^d. 
$$
It can also be written 
$$
(\widehat{P\boldsymbol{u}})_i = \sum_{j = 1}^3 (\delta_{ij} - A_{\Phi_{ij}}) \widehat u_j, 
$$
with $\Phi_{ij}(0) = 0$, $\Phi_{ij}$ of class $\mathcal{C^\infty}$ and $\Phi_{ij}(x) = \frac{x_{i}x_j}{|x|^2}$, for $ |x| > \frac12$. Hence we see that $P$ can be identified with a $3\times3$ matrix of infinite dimensional matrices belonging to $\Ac_{\boldsymbol{0}}$ as all the components of the matrix operator have order $0$.  

More generally, all the system of differential operators can be expressed as elements of $\Ac_{\boldsymbol{r}}$ for some matrix of orders, for instance elliptic system of Agmon, Douglis and Nirenberg type \cite{ADN59} or hyperbolic systems such as system of conservation laws \cite{Bre00}. 

\subsection{Geometric conditions\label{geosec}}

In this section we  consider subspaces of $\Ac_r$ that are stable by bracket and define Lie algebras. In these case, the flow operator is well defined and belong to an infinite dimensional Lie group. We give in sections \ref{hermi} and \ref{sectionsymp} two examples but the reader can imagine many other situations. \\
First we show in section \ref{dir} how we can encode different type of boundary condition in the matrix algebra $\Ac$. 

\subsubsection{Dirichlet boundary condition}\label{dir}

On periodic functions, we can easily consider Dirichlet or Neumann boundary conditions by imposing some parity conditions. For example if $u: \T^d \to \C$ satisfies $u(-x) = -u(x)$, then $\hat u_k = \hat u_{-k}$ for $k \in \Z^d$, and $u$ vanishes on the boundary of $\T^d$ represented as $[0,2\pi]^d$, {\em i.e.} $u$ satisfies Dirichlet boundary conditions on this domain. 

In order to be able to consider operators preserving this boundary conditions, we define $\Ac^D$ the subclass of $\Ac$ defined by
$$A\in\Ac^D \iff A\in\Ac \quad\text{and}\quad A(-m,-n)=A(m,n),\ \forall m,n\in\Z^d.$$
We also define $h^s_{\rm odd}$ as the subspace of $h^s$ formed by the odd sequences:
$$x\in h^s_{\rm odd}\iff x\in h^s \quad\text{and}\quad x_{-k}=-x_k,\ k\in\Z^d.$$
Matrices in $\Ac^D$ preserves oddness of sequences and is stable by multiplication and bracket. Thus as a consequence of Lemma \ref{op} we get that every matrix $A\in\Ac_r^D$ naturally defines a continuous operator, still denoted by $A$, from $h^s_{\rm odd}$ to $h^{s-r}_{\rm odd}$ for any $s\in\R$ via the formula \eqref{Ax} and we have
 $$\Norm{ A}{\Lc(h^s_{\rm odd},h^{s-r}_{\rm odd})}\leq C \Norm{A}{0,|s|+|r|+d+1,r}$$
 for some constant $C$ depending only on $d$.

We can of course also consider Neumann boundary conditions or mixed boundary conditions. 

\subsubsection{Hermitian operators}\label{hermi}

In order to consider equations of Schr\"odinger form, we  define the set of Hermitian operators as the set of  $H\in\Ac$  satisfying 
\begin{equation}
\label{adjoint}
 H(m,n) = H^*(m,n) := \overline{H(n,m)}, \quad m,n \in \Z^d. 
\end{equation}
With obvious notations, we write this conditions $H = H^* :=  \overline{H}^T$ where the transpose matrix is defined by exchanging $m$ and $n$ in the coefficients. 
It is easy to check that for all $\Phi$ real-valued, the operator $H = A_\Phi$ is diagonal and Hermitian, and that $H= B_V$ is also Hermitian when $V$ is a real function. More generally, representation of operator of the form 
\begin{equation}
\label{order2}
u \mapsto \mathrm{div} (\sigma(x) \nabla u ) + X(x) \cdot\nabla u -\mathrm{div} (X(x) u) + V(x) u
\end{equation}
for smooth vector field $X$ with $\mathrm{div} X = 0$, and real functions $\sigma$ and $V$, yields to Hermitian operators. 

In section \ref{precon}, we will consider the Hermitian operator $H = - \Delta + V$ for a smooth real function $V(x)$ on $\T^d$ and numerical schemes to approximate the solutions of the associate Schr\"odinger equation $i \partial_t u = (- \Delta + V)  u$ .

\subsubsection{Symplectic matrices\label{sectionsymp}}

 
We  define symplectic systems as follows: for real matrices $A$, $B$ and $C$ ({\em i.e.} matrices with real coefficients) we set 
\begin{equation}
\label{defS}\boldsymbol{S}=\begin{pmatrix} A & B \\ C & -A^T \end{pmatrix} \quad \mbox{with} \quad B^T = B \quad \mbox{and}\quad C^T = C,
\end{equation} which is an element of $\Ac_{\boldsymbol{r}}$ (see Definition \eqref{defsys}) with $\boldsymbol{r} = \begin{pmatrix} r(A) & r(B) \\ r(C) & r(A) \end{pmatrix}$ where $r(A)$, $r(B)$ and $r(C)$ denote the orders of $A$, $B$ and $C$ respectively. 
To $\boldsymbol{S}$ we associate the symplectic system on $ \ell^2 \times \ell^2\ni(p,q)$
\begin{equation}
\label{eqsymp}
\frac{\dd}{\dd t} \begin{pmatrix}
p \\ q 
\end{pmatrix} = \boldsymbol{S} \begin{pmatrix}
p \\ q 
\end{pmatrix}, \qquad \boldsymbol{S} = \begin{pmatrix} A & B \\ C & -A^T \end{pmatrix} 
\end{equation}
and, when it is well defined\footnote{Using typically the fact that the unbounded part can be diagonalized explicitly in Fourier to define mild-solutions, an example is given below.}, its flow $e^{t \boldsymbol{S}}$  acting on $\ell^2 \times \ell^2$. This symplectic flow  preserves the canonical symplectic form: defining  $J = \begin{pmatrix} 0 & I \\-I & 0 \end{pmatrix}$ we have 
$$
(e^{ t \boldsymbol{S}})^T J e^{ t \boldsymbol{S}} = J. 
$$
A typical example is given by 
wave equations of the form $\partial_{tt} q - \Delta q = V(x) q$ that can be written 
$$
\frac{\dd}{\dd t}\begin{pmatrix}
p \\ q 
\end{pmatrix}
= \begin{pmatrix}
0 & \Delta + V \\ I & 0
\end{pmatrix}
\begin{pmatrix}
p \\ q 
\end{pmatrix}
$$
with $p = \partial_t$, where the right-hand sides defines a $2\times 2$ system of order $\begin{pmatrix} 0 & 2 \\ 0 & 0 \end{pmatrix}$. Note however that this system can be also reformulated by using pseudo-differential transformations as a skew symmetric system of orders $\leq 1$. We will give explicit examples in Section \eqref{water}. 

\section{Periodic matrices and discretization}\label{per}
Let $K$ be an even integer\footnote{Working with arbitrary integers is of course possible, by changing the structure of $G_K$ according to the parity of $K$, see \cite{Fao12}} and we define the grid points 
\begin{equation}\label{GK}
x_a = \frac{2\pi a}{K}, \quad a \in \{ -K/2, \ldots, K/2 - 1\}^d:=G_K
\end{equation}
and we identify $G_K$ with $(\Z /K\Z)^d := \Z_K^d$ the set of equivalent class modulo $K$ in each variable:  to each $a\in\Z_{K}^d$ we associate  $\hat a$ its unique representative  within $G_K$.
We set $h = \frac{2\pi}{K}$. The grid $x_a$ can thus be written $x_a = a h \in \T^d$, $a \in \Z_K^d$. When this grid is used to discretize a function $u: \T^d \to \C$, we expect to approach $u(ah) \simeq u_a$, $a \in \Z_K^d$. Hence the function space is discretized by $u = (u_a)_{a \in \Z_K^d} \in \ell^2(\Z_K^d)$\footnote{Notice that $\ell^2(\Z_K^d)$ is finite dimensional space equivalent to $\C^{K^d}$, we choose to equipped it with the $\ell_2$-norm.}. 
Very schematically a linear numerical scheme with mesh $h = \frac{2\pi}{K}$ is an  application 
$$\ell^2(\Z_K^d)\ni(u^{K}_a)_{a \in \Z_K^d}\mapsto (v^{K}_a)_{a \in \Z_K^d}\in\ell^2(\Z_K^d).$$
which can be represented by a periodic matrix $M^{K}$:
$$(v^{K}_a)_{a \in \Z_K^d}=M^{K}(u^{K}_a)_{a \in \Z_K^d}.$$
We thus naturally see the need of a concept of a {\em family} of periodic matrices $M^K$ indexed by $K$, the number of points in our grid (or equivalently in the periodic case, by $h$ the mesh size). Of course, as finite dimensional matrix, the norm of $M^K$ is bounded but depends {\em a priori} on $K$. 
In order to evaluate the convergence of the scheme, or to study the global properties of numerical schemes at the continuous limit $K \to +\infty$, it will be essential to have norms on these families of matrices which are {\em uniform} in $K$. Moreover, the notion of pseudo-differential operator is well expressed in term of Fourier transform. Hence in the discrete case, we expect the Fourier transformation 
$$
A^K = \Fc_K^{-1} M^K \Fc_K
$$
to inherit the commutator properties of continuous systems, uniformly in $K$. Here $\Fc_K$ stands for the discrete Fourier transform, see \eqref{TFK}. 
  This justifies the space of families of periodic matrices that we will define in the next section.

\subsection{A class of families of periodic matrices}
\begin{notation}
Let $K \in 2 \N^*$ an even integer. 
For $a\in\Z_{K}^d$,  we denote by $\hat a$ its  representative  within $\{-K/2, \ldots, K/2 - 1\}^d := G_K$ and we set $[a]=|\hat a_1|+\cdots+|\hat a_d|$.  
\end{notation}
This quantity has some good properties: 
\begin{lemma}\label{classe}
For all $a,b,c\in\Z_{K}^d$ we have
\begin{itemize}
\item[(i)] $[a+b]\leq [a]+[b]$,
\item[(ii)] $(1+[a]+[c])\leq 2(1+[a]+[b])(1+[c-b])$.
\end{itemize}
\end{lemma}
\begin{proof} {\em (i)} It suffices to consider the case $d=1$. If $\hat a +\hat b \in G_K = \{-K/2, \ldots, K/2 - 1\}$ then $\widehat{a+b}=\hat a+\hat b$ and thus $|\widehat{a+b}|\leq|\hat a|+|\hat b|$.  If $\hat a +\hat b \notin G_K$ then $[a+b] = |\widehat{a + b}| \leq K/2   \leq|\hat a +\hat b|\leq [a]+[b]$.\\
{\em (ii)} is an easy consequence of {\em (i)}. 
\end{proof}
We will consider familly of matrices $A^K: \ \Z_K^d\times\Z_K^d\mapsto \C$, indexed by $K$. For one give $K$, the matrix $A^K$ is identified with the collection of its complex elements $A^K := \{A^K(m,n)\}_{(m,n) \in \Z_K^d \times \Z_K^d}$, where $A^K(m,n)$ is now $K$-periodic in $m$ and $n$. For such a matrix, we extend easily the definition of $\Delta_j^+$, $\Delta_j^-$ and $\Delta^\alpha$ by periodicity. 

\begin{definition}
\label{defper}
Let $r\in\R$. We define the class $\Ac_r^{\rm per}$ of  families of periodic matrices   of order $r$ as follows: the family $A^\bullet=\{A^K\}_{K\in2\N^*}$ belongs to $\Ac_r^{\rm per}$ if  for all $\alpha\in\Z^d$  and all $N\in\N$ there exists $C_{N,\alpha}>0$ such that
\begin{equation}
\label{AKbound}
 \forall K \in \N, \quad \forall m,n\in\Z_{K}^d, \quad \big| (\Delta^\alpha A^K)(m,n)\big|\leq C_{N,\alpha}(1+[m]+[n])^{r-|\alpha|}(1+[m-n])^{-N}.
 \end{equation}
\end{definition}
We denote  $\Ac^{\rm per}=\cup_{r\in\R}\Ac^{\rm per}_r$ which form a graded algebra.
We define the adapted  family of semi-norms: for $A = \{A^K\}_{K \in 2\N^*}$, 
$$\DNorm{A^\bullet}{\alpha,N,r}:=\sup_{K\in\N} \sup_{m,n\in \Z_{K}^d}\frac{\big| (\Delta^\alpha A^K)(m,n)\big| (1+[m-n])^{N}}{(1+[m]+[n])^{r-|\alpha|}}.$$
\begin{remark}
In Equation \eqref{AKbound} it is important to notice that the matrices $K$ are of size $K \times K$, and the norm $[a]$ also depend on $K$. However, the key property is that $N$, $\alpha$ and the constant $C_{N,\alpha}$ are assumed to be uniform in $K$. 
\end{remark}

We then define the product and commutator as follows: for $A^\bullet = \{A^K\}_{K \in 2\N*}$ and $B^\bullet = \{ B^K\}_{K \in 2\N^*}$ in $\Ac^{\rm per}$, we set 
$$
A^\bullet B^\bullet = \{ A^K B^K\}_{K \in 2 \N^*} \quad \mbox{and} \quad [A^\bullet,B^\bullet] = \{\, [A^K,B^K]\,\}_{K \in 2 \N^*}. 
$$
Thanks to Lemma \ref{classe}, the proofs of Lemmas \ref{prod} and \ref{com} are transposable, {\em mutatis mutandis}, to the periodic case. Thus we obtain 
\begin{proposition}\label{Prop-Struct-Per} Let $r_1,r_2\in\R$.
\begin{itemize} 
\item[(i)] The matrix product is a continuous map from $\Ac^{\rm per}_{r_1}\times\Ac^{\rm per}_{r_2}\ni (A^\bullet,B^\bullet)\mapsto A^\bullet B^\bullet\in \Ac^{\rm per}_{r_1+r_2}$. More precisely, for all $\alpha\in\Z^d$  and all $N\in\N$ there exists $C(\alpha,N,r_1,r_2)>0$, independent of $K$, such that
$$\DNorm{ A^\bullet B^\bullet}{\alpha,N,r_1+r_2}\leq C(\alpha,N,r_1,r_2)\Big(\sum_{\beta\in [0,\alpha]} \DNorm{ A^\bullet}{\beta,N+1+|r_2|,r_1}\Big)\Big(\sum_{\beta\in [0,\alpha]} \DNorm{ B^\bullet}{\beta,N+1+|r_1|,r_2}\Big).$$
\item[(ii)] The commutator gains one order : the map $\Ac^{\rm per}_{r_1}\times\Ac^{\rm per}_{r_2}\ni(A,B)\mapsto [A^\bullet,B^\bullet]\in \Ac^{\rm per}_{r_1+r_2-1}$ is uniformly continuous in $K$. More precisely, for all $\alpha\in\Z^d$  and all $N\in\N$ there exists $C(\alpha,N,r_1,r_2)>0$, independent of $K$, such that
\begin{equation}\label{-1K}\DNorm{\,  [A^\bullet,B^\bullet]\, }{\alpha,N,r_1+r_2-1}\leq C(\alpha,N,r_1,r_2)\Big(\sum_{\substack{\beta,\gamma\in [0,\alpha]\\ M_1,M_2\leq 2N+|r_1|+|r_2|}} \DNorm{A^\bullet}{\beta,M_1,r_1} \DNorm{B^\bullet}{\gamma,M_2,r_2}\Big).\end{equation}
\end{itemize}
\end{proposition}
Families of matrices $A^\bullet = \{A^K\}_{K \in 2\N^*}$ in $\Ac^{\rm per}$ define naturally families of finite dimensional operators $A^K$ on 
$$
\ell^2(\Z_K^d) :=\{(x_k)_{k\in\Z^d}\in\C^{\Z^d}\mid x_{k}=x_j \text{ when } k_i\equiv  j_i  \;\mathrm{mod}\; K \text{ for }i=1,\cdots,d\}= \C^{\Z^d_K}. 
$$
for each $K\in2\N^*$ via the formula 
$$
(A^Kx)_a=\sum_{k\in\Z_{K}^d}A^K(a,k)x_k, \quad a\in \Z_{K}^d.
$$
The space $\ell^2(\Z_K^d)$ has finite dimension nevertheless Lemma \ref{op} has an interesting counterpart in the periodic case: for $x\in \ell^2(\Z_K^d)$ we denote 
$$
\Norm{x}{s,K}^2=\sum_{k\in \Z_{K}^d}(1+[k])^{2s}|x_k|^2, 
$$
and hence $\Norm{x}{\ell^2(\Z_K^d)} = \Norm{x}{0,K}$. 
Of course, since $\ell^2(\Z_K^d)$ has finite dimension, all the norms $\Norm{\,\cdot\,}{s,K}$ are equivalent on $\C^{\Z^d_K}$, but with constant depending on $K$, for instance 
$$
\Norm{x}{s,K} \leq (1 + |K|)^{s} \Norm{x}{0,K}
$$
however, with the help of Definition \ref{defper}, we have with the same proof as Lemma \ref{op}
\begin{lemma}\label{op-per}
Let $r,s\in\R$ and  $A^\bullet = \{A^K\}_{K \in 2\N^*}\in\Ac^{\rm per}_r$  we have
 $$
 \forall\, K \in 2\N^*\quad \Norm{ A^Kx}{s-r,K}\leq C \DNorm{A^\bullet}{0,|s|+|r|+d+1,r}\Norm{x}{s,K}$$
 for some constant $C$ depending only on $d$ (and thus independent of $K$).
\end{lemma}

\subsection{Representation of finite different schemes}\label{fds}

For $j = 1,\ldots,d$, we  define the  finite difference operators from $\ell^2(\Z_K^d)$ into $\ell^2(\Z_K^d)$:
$$
(\delta^+_{j,K}  u)_a = \frac{ u_{a+e_j} - u_a}{h} \quad \mbox{and} \quad 
(\delta^-_{j,K}  u)_a = \frac{ u_{a} - u_{a- e_j}}{h} , \quad i = 1, \ldots, d, \quad h = \frac{2\pi}{K}
$$
Another important tool is the discrete Fourier transform: $\Fc_K: \ell^2(\Z_K^d) \to \ell^2(\Z_K^d)$ such that for all $v = (v_a)_{a \in \Z_K^d} \in \ell^2(\Z_K^d)$, 
\begin{equation}
\label{TFK}
(\Fc_K v)_a = \frac{1}{K^d} \sum_{b \in \Z^d_K} e^{- \frac{2 i \pi a \cdot b}{K}} v_b= 
\frac{1}{K^d} \sum_{b \in \Z^d_K} e^{-i  a \cdot x_b} v_b. 
\end{equation}
It inverse is given by 
$$
(\Fc_K^{-1} v)_a =  
\sum_{b \in \Z^d_K} e^{\frac{2 i \pi a \cdot b}{K}} v_b = (K^d \Fc^* v)_a. 
$$
It is well known that the transformation $K^{d/2} \Fc_K$ is unitary, and that this transformation can be efficiently implemented by using Fast Fourier Transform algorithms.

\subsubsection{Difference operators}
In a finite dimensional setting, the operators $h\delta_{j,K}^{\pm}$ are represented by a matrices with  $\mp1$ on the diagonal and $\pm1$ on one of the first diagonals, with complementary coefficient $\pm1$ in the corner to ensure the periodicity. For instance when $d=1$ we have
$$\delta_{1,K}^+= \frac{1}{h}\left(\begin{array}{ccccc}-1 & 1&0&\cdots& 0\\
0&-1&1&\cdots&0\\
 \vdots& & \ddots&\ddots &  \\
 0&  & &-1&1\\
1&0&\cdots&0&-1 \end{array}\right).$$
Let us remark that this family of matrices, when $K \in 2\N^*$, does not define an element of any  $\Ac_r^{\rm per}$, $r \in \R$. In particular because of the entry of index $(0,0)$ equals $\mp\frac{K}{2\pi}$ which cannot be estimated independently of $K$ as it should be. Nevertheless we are going to prove that in the Fourier side these operators are in $\Ac_1^{\rm per}$ (and diagonal).

Let $Q_K$ be the matrix associated with the discrete fourier transform $K^{d/2} \Fc_K$: 
$$
Q_K(a,b) = K^{-d/2}e^{- 2i\pi a\cdot b/K} = \Big(\frac{ h}{2\pi}\Big)^{d/2}e^{- i h a \cdot b }. 
$$
By using the aliasing formula
$$
\Big(\frac{ h}{2\pi}\Big)^{d}\sum_{b \in \Z_K^d} e^{h i j \cdot b} = \left\{
\begin{array}{l}
1 \quad \mbox{if} \quad j = mK, \quad m \in \Z^d\\[1ex]
0\quad \mbox{else}, 
\end{array}
\right.
$$
we see that $Q_K$ is unitary, {\em i.e.} $Q_K^* Q_K = 1$, see \eqref{adjoint}. 
Moreover, we have the following Lemma: 
\begin{lemma}
Let $u = (u_a)_{a \in \Z_K^d}$ and $\hat u = \Fc_K u= (\hat u_b)_{b \in \Z_K^d}$. Then we have for $j = 1,\ldots,d$, 
$$
\left|
\begin{array}{l}
\delta_{j,K}^{+} =  Q_K^* D^{+}_{j,K}  Q_K = \Fc_K^{-1} D^{+}_{j,K}  \Fc_K\quad\mbox{and}\\[1ex]
\delta_{j,K}^{-} =  Q_K^* D^{-}_{j,K}  Q_K = \Fc_K^{-1} D^{-}_{j,K}  \Fc_K
\end{array}
\right.
$$
where $D^+_{j,K}$ and $D^-_{j,K}$ are the diagonal operators 
$$
\left|
\begin{array}{rcll}
(D^+_{j,K} \hat u)_a &=& \frac{1}{h} (e^{i h a_j } - 1) \hat u_a, & a \in \Z_K^d, 
\quad\mbox{and}\\[1ex]
(D^+_{j,K} \hat u)_a &=& \frac{1}{h} (1 - e^{- i h a_j }) \hat u_a, & a \in Z_K^d. 
\end{array}
\right.
$$
Furthermore, the matrices $\{D^+_{j,K}\}_{K \in 2\N*}$ and $\{D^{-}_{j,K}\}_{K \in 2\N*}$ are in $\Ac_1^{\rm per}$. 
\end{lemma}
\begin{proof}
We have $u = \Fc_K^{-1} \hat u$ which is written in coordinates 
$$
u_a = \sum_{b \in \Z_K} e^{\frac{2 i \pi a \cdot b}{K}} \hat u_b. 
$$
Hence 
$$
u_{a + e_j} = \sum_{b \in \Z_K^d} e^{\frac{2 i \pi a \cdot b}{K}} e^{\frac{2 i \pi e_j \cdot b}{K}}\hat u_b = \sum_{b \in \Z_K^d} e^{\frac{2 i \pi a \cdot b}{K}} e^{\frac{2 i \pi b_j}{K}} \hat u_b . 
$$
This shows that 
$$
(u_{a + e_j}  - u_a)_{a \in \Z_K^d} = \Fc_K^{-1} \left\{ (  (e^{\frac{2 i \pi b_j}{K}} - 1) \hat u_b)_{b \in \Z_K^d}  \right\}
$$
and, as $h = \frac{2\pi}{K}$, we obtain the representation of $\delta_{j,K}^{+}$ given in the Lemma.   \\
To prove that $\{D^+_{j,K}\}_{K \in 2\N*}\in\Ac_1^{\rm per}$, we note that the components of the matrix $D^+_{j,K}$ satisfy 
\begin{equation}\label{DjK}
D^{+}_{j,K} (a,b) = \frac{1}{h} (e^{ih a_j}- 1)  \delta_{a,b}, \quad a,b \in \Z_K^d. 
\end{equation}
But we as $h = 2\pi/K$, we have $e^{ih a_j} = e^{i h \hat a_j}$. 
Hence using the fact that $|e^{ix} - 1| \leq  |x|$ for real numbers $x$, we have 
$$
|D^{+}_{j,K} (a,a) | \leq  |\hat a_j |\leq [a]\quad \mbox{as} \quad a \in \Z_K^d. 
$$
Moreover, we have for $\alpha \geq 1$, 
$$
\partial_{a_k}^\alpha D^{+}_{j,K} (a,a) = i^\alpha h^{\alpha - 1} \delta_{kj} e^{i h a_j}. 
$$
By using Taylor expansion, we thus have, as in \eqref{estDA}, that for $\alpha \in \Z^d\setminus\{0\}$, 
$$
|\Delta^\alpha D^{+}_{j,K} (a,a) | \lesssim_\alpha  h^{1 - |\alpha|}  \lesssim \frac{1}{K^{|\alpha| - 1}} \lesssim ( 1 + [a])^{1 - |\alpha|}
$$
as we always have $[a] \leq K$. Therefore the familly $\{D^+_{j,K}\}_{K \in 2\N^*} \in\Ac_1^{\rm per}$. 
\end{proof}

\subsubsection{Pointwise Multiplication}\label{mult}
We consider now a periodic function $V(x)$, $x \in \T^d$. For $k \in \Z^d$, we denote by $(\Fc V)_k$ the Fourier transform of $V$. We thus have 
$$
V(x) = \sum_{k \in \Z^d} (\Fc V)_k e^{i k \cdot x}. 
$$
A natural discretization of the operator $u(x) \mapsto V(x) u(x)$ on a grid consists in the pointwise multiplication  
$$
(u_a)_{a \in \Z_K^d} \mapsto (V_a u_a)_{a \in \Z_K^d} =: \{(B_{V,K} u)_a\}_{a \in \Z_K^d}
$$
where $V_a = V(ah)$ defines a element of $\ell^2(\Z_K^d)$. Denoting by $\hat V_a = \Fc_K \{ (V_a)_{a \in \Z_K^d}\}$ the discrete Fourier transform of the sequence $V_a$, we thus have 
\begin{equation}\label{certes}
V_a = V(ah) = \sum_{b \in \Z_K} \hat V_b e^{i h b \cdot a } = \sum_{k \in \Z^d} (\Fc V)_k e^{i h k \cdot a}. 
\end{equation}
In particular this shows that 
$$
\hat V_b = \sum_{\ell\in\Z^d} (\Fc V)_{b + \ell K}.
$$
Using \eqref{certes} for $V_a$ and $u_a$ we get
$$
V_a u_a = 
\sum_{b,c \in \Z_K^d} \hat V_b e^{i (b + c) \cdot a h} \hat u_c =\sum_{f  \in \Z_K^d} e^{i a \cdot f h} \Big( \sum_{b + c = f} \hat V_b\hat u_c\Big)
$$
which leads to
$$
\Fc_K \{ (V_b u_b)_{b \in \Z_K^d} \}_a  = \Big( \sum_{b \in \Z_K^d} \hat V_{a-b}\hat u_b \Big)_a.
$$
As in the previous section, we obtain the following result: 
\begin{lemma}
\label{lemalias}
For $K \in 2 \N$, 
let $u = (u_a)_{a \in \Z_K^d}$ and $\hat u = \Fc_K u$ and we denote $\hat u = (\hat u_b)_{b \in \Z_K^d}$ its components. Let $V: \T^d \to \C$ be a smooth periodic function, and let $B_{V,K}$ be the diagonal operator acting on $u$ defined by $(B_{V,K} u)_a = V_a u_a$ where $V_a = V(ah)$, $a\in  \Z_K^d$. 
Then we have 
$$
B_{V,K} =  \Fc_K^{-1} M_{V,K} \Fc_K
$$
where $M_{V,K}$ is the matrix with entries
\begin{equation}\label{MVK}
M_{V,K}(a,b) = \hat V_{a -b} = \sum_{\ell \in \Z^d} (\Fc V)_{a - b + \ell K}, \ a,b\in\Z_K^d
\end{equation}
where $(\Fc V)_k$ denote the coefficients of the Fourier transform of the periodic function $V(x)$.\\ 
Furthermore the family of matrices $M_{V,\bullet}:=\{M_{V,K}\}_{K\in2\N}$ belongs to $\Ac_0^{\rm per}$. 
\end{lemma}
\begin{proof}
Only the last statement has not been proven. As a consequence of the smoothness of $V$ we have that for all $M$ there exits $C_M$ such that
$$
\forall\, j \in \Z^d, \quad |(\Fc V)_{j}| \leq C_M \frac{1}{(1 + |j|)^M}. 
$$
Hence 
$$
|\hat V_{a -b}| \leq C_M \sum_{\ell \in \Z^d} \frac{1}{(a - b + \ell K)^M } 
= C_M \left( \frac{1}{(1 + |a - b|)^M} + \sum_{\ell \neq 0} \frac{1}{(1 + |a - b + \ell K|)^M }\right). 
$$
To bound the first term in the right-hand side, we note that when $a, b \in \Z_K^d$, then either $a - b \in G_K = \{ -K/2, \ldots, K/2-1\}$ and then $|a - b| = [a - b]$, or  $a - b \notin G_K$, and then we have $|a - b| \geq K/2 \geq [a-b]$. In both case, we have $1 + |a - b| \geq 1 + [a - b]$. To deal with the second term, we note that 
when $|\ell| \geq 2$, we have $a - b + \ell K \geq K/2 \geq [a-b]$. Hence the sum of these terms is bounded by 
$$
C_M \frac{1}{(1 + [a-b])^{M-d}} \sum_{|\ell| \geq  2} \frac{1}{(1 + |a - b + \ell K|)^d }
$$
and the last sum converges independently of $K$, $a$ and $b$. 
It remains to consider the case $\ell = \pm 1$. In the case $a - b \in G_K$, we have as before $|a - b| \geq K/2 \geq [a-b]$ and  hence we can control the term by  $\frac{1}{(1 + [a-b])^M}$. Finally, when $a - b \notin G_K$ we are in a situation where $a - b \pm K = [a - b]$ while $a - b \mp K \geq K/2 \geq [a - b]$ and we can control the term in both situations.  
This shows that for all $N$, 
$$
|\hat V_{a-b}|\leq C_N \frac{1}{(1 + [a-b])^N}
$$ 
We conclude by noticing that as in the infinite case, we have $\Delta^\pm_{j} (V_{a-b}) = 0$. 
\end{proof}

\subsubsection{General finite difference operators}

By arguing as in Corollary \ref{cor28}, we prove the following result: 
\begin{corollary}
Let $K \in 2 \N^*$, $u = (u_a)_{a \in \Z_K^d}$ a sequence, and $\hat u = \Fc_K u$ its discrete Fourier transform. Any composition of difference operators $\delta_{j,K}^\pm$ and multiplication operators of the form $B_{V,K}$ for some smooth periodic functions $V$ acting on $u$, defines a discrete pseudo differential operator acting on $\hat u$. More precisely, for $P \in \N$ and for $p = 1,\ldots,P$  let $V_p$  be smooth functions, and $\epsilon_p \in \{0,\pm\}$. Then for all $K$ we  define
\begin{equation}
\label{mulAk}
A^K = \prod_{p = 1}^P M_{V_p,K} D^{\epsilon_p}_{j,K}, 
\end{equation}
with the convention $D^{0}_{j,K} = \mathrm{Id}_K$, and 
 the family $A^\bullet = \{A^K\}_{K \in 2 \N^*}$ defines an element of $\Ac_r^{\rm per}$ with $r = \sum_{p = 1}^P |\epsilon_p|$. 
\end{corollary}

Note that the decomposition \eqref{mulAk} correspond after Fourier transform to matrix the decomposition 
$$
\Fc_K A^K \Fc_K^{-1} = \prod_{p = 1}^P B_{V_p,K} \delta^{\epsilon_p}_{j,K}, 
$$
acting on $u$. 
For example, a discretization of the order 2 operator \eqref{order2} can be written (for example)
\begin{equation}
u_a \mapsto \sum_{j = 1}^d \delta_j^+ B_{\sigma,V} \delta_j^- u   + B_{X_j,K}  \delta_j^+ 
 - \sum_{j} \delta^+_j (B_{X_j,K} u) +  B_{V,K} u, 
\end{equation}
Which yields a familly of operators of order $2$, belonging to the class $\Ac_2^{\rm per}$. 

Similarly, all finite difference operators discretizing transport equation (like upwind, WENO schemes, $\ldots$) fall into the same category. Note that geometric considerations as in Section \eqref{geosec} can be made, according to the situation, the basic tool being given by the two previous Lemmas. 

\subsection{Pseudo-spectral methods}\label{pseudo}

When discretizing general pseudo-differential equations, we face the problem of approximating operators of the form $\Phi(- i \partial_x)$ when $\Phi$ is not a polynomial but can be a rational or any function  (see \eqref{ww} in the case of water waves). 
In this case, one possibility is to consider spectral methods: A discretization of $u \mapsto \Phi(-i \partial_x) u$ is directly written in the discrete Fourier space
\begin{equation}\label{APK}
\hat u_a \mapsto \Phi(a) \hat u_a =: (A_{\Phi,K}\hat u)_a, \qquad a \in \Z_K^d. 
\end{equation}
This yields to the evaluation of a diagonal operator in Fourier. 

Now to approximate an operator of the for $u(x) \mapsto V(x) u(x)$, pseudo-spectral methods consist in calculating 
$$
\hat u \mapsto \Fc_K B_{V,K}\Fc_K^{-1} \hat u
$$
where we recall that the operator $B_{V,K}$ is the pointwise multiplication by $V(ah)$ on the grid points $ah$, $a \in \Z_K^d$ (see section \ref{mult}). Note that the evaluation of $ \Fc_K B_{V,K}\Fc_K^{-1}$ doesn't cost too much since  the operator $B_{V,K}$ is diagonal and the discrete Fourier transforms $\Fc_K$ and $\Fc_K^{-1}$  can be efficiently implemented by Fast Fourier transform algorithm. 

This way, pseudo-spectral methods are efficient algorithm to discretize efficiently operators of the form \eqref{AAA}. In echo with Lemma \ref{phietv}, we can state the following result
\begin{lemma}
\label{phietvK}
Let $\Phi: \C^d \to \C$ and $V: \T^d \to \C$. Then 
\begin{itemize}
\item[(i)]
If $\Phi$ is $\mathcal{C}^\infty$ and if there exists $r \in \R$ such that 
for all $\alpha \in \N^d$ and $x \in \R^d$,  $|\partial_x^\alpha \Phi (x)| \lesssim_{r,\alpha} \langle x\rangle^{r - |\alpha|}$, then the family of matrices $A_{\Phi,\bullet} = \{A_{\Phi,K}\}_{K\in2\N^*} \in \Ac_r^{\rm per}$. 
\item[(ii)] If $V$ is $\mathcal{C}^\infty$ then the family $\{\Fc_K B_{V,K}\Fc_K^{-1}\}_{K\in2\N^*}  \in \Ac_0^{\rm per}$. 
\end{itemize}
As a corollary, pseudo-spectral discretization of compositions of the form \eqref{AAA} belong to $\Ac_{r}^K$, with $r$ as in Corollary \ref{cor28}. 
\end{lemma}
The proof is based on the same argument as the proof of Lemma \eqref{phietv} combined with the aliasing calculation of Lemma \eqref{lemalias}. 

\subsection{Approximation issues}

As explained in the introduction of Section \ref{per}, one of the main motivation in the introduction of families of periodic matrices $A^\bullet =  \{A^K\}_{K \in 2\N^*}$ in $\Ac^{\rm per}_r$ is to consider discretizations of PDEs, set on $\T^d$, on a grid of mesh $h=\frac{2\pi}K$. So naturally we are interested in the limit $K\to \infty$. \\
If we fix $K$, the matrix $A^K$ can be  embedded into $\Ac_r$ as an infinite dimensional operator that we denote by $\mathcal{I}A^K$ and simply defined as (see \eqref{GK})
\begin{equation}
\label{embed}
\mathcal{I}A^K(m,n) = 
\left|
\begin{array}{ll}
A^K(m,n) & \mbox{if}\quad  (m,n) \in  G_K \times G_K\\[1ex]
0 & \mbox{else}. 
\end{array}
\right.
\end{equation}
Then we can expect the convergence of $A^K$ towards some infinite dimensional operator $A \in \Ac_r$.
However, in general, $\Norm{\mathcal{I}A^K}{\alpha,r,N}$ cannot be controlled uniformly by 
$\DNorm{A^\bullet}{\alpha,r,N}$, 
as we can have $|m - n| > [m-n]$ even when $m,n \in \Z_K^d$. In fact the matrix was originally periodic and thus in the box  $G_K\times G_K$  the top right corner is identified with the top left corner (consider the cas $d=1$). In other words  there are possible large coefficients $\mathcal{I}A^K(m,n)$ when typically $m = -K/2$ and $n = K/2 - 1$ where $|m-n|=K-1$ but $[m-n]=1$. It is a typical {\em aliasing} phenomenon. For this reason, 
the convergence of $A^K$ towards some infinite dimensional operator $A \in \Ac_r$ does not hold in the norm of $\Ac_r$ after using the embedding $\mathcal{I}$, but has to be understood in a weaker sense.

To illustrate this phenomenon, let us consider the approximation of
 the operator $B_V$ given in \eqref{BV} by the family of periodic matrices $M_{V,\bullet}$ defined in \eqref{MVK}. We have 
$$
(\mathcal{I} M_{V,K}) (m,n) - B_V(m,n) = \left|
\begin{array}{ll}
 \sum_{\ell \neq 0} (\Fc V)_{m - n  + \ell K} & \mbox{if} \quad (m,n) \in G_K \times G_K \\[1ex]
(\Fc V)_{m-n} & \mbox{else}
\end{array}
\right.
$$
We see again that as in the previous example, we do not have that $\Norm{\mathcal{I} M_{V,K} - B_V}{\alpha,N,0}$ goes to zero, when $K \to \infty$. This can be seen by considering the entry $(m,n)$  with $ m = -K/2 $ and $n = K/2-1$. This coefficient is equal to $ \sum_{\ell \neq 0} (\Fc V)_{-K+1  + \ell K} $ which is not small in general since it contains $(\Fc V)_1$ ( for $\ell = 1$). However, we have the classical estimate with loss: for $0\leq s'<s$ (see for instance \cite{Lub08,Fao12})
$$
\Norm{(\mathcal{I} M_{V,K})  - B_V) x}{s'}  \leq \frac{C}{K^{s - s'}} \Norm{x}{s}.
$$

Convergence issues can also appear because of the order of the operators we are trying to approximate. For instance let us consider the family of diagonal operators $D^+_{j,\bullet} = \{ D^+_{j,K} \}_{K \in 2 \N^*}$ defined in  \eqref{DjK}. In that case the previous aliasing phenomenon does not occur since the matrices are diagonal and for all $\alpha \in \Z^d$ and $N \in \N$ there exists $C$ such that  
$$
\forall K\in 2\N^*, \quad 
\Norm{\mathcal{I}D^+_{j,K} }{\alpha,N,1} \leq C \DNorm{D^+_{j,\bullet} }{\alpha,N,1}. 
$$
On the other hand the limit operator of the $D_{j,K}^+$, when $K  \to \infty$ is naturally the operator $A_{\Phi_j}$ with $\Phi(x) = ix_j$, for $j \in \{1,\ldots,d\}$ and $x = (x_1, \ldots, x_d)$. But the function $m\mapsto i m - \frac{1}{h}( e^{ih m} - 1)$ does not go to $0$ uniformly in $h$. Hence $\Norm{A_{\Phi_j} - \mathcal{I} D^+_{j,K}}{\alpha,N,1}$ does not goes to $0$ when $K \to \infty$. However, using Taylor expansion, we obtain easily the classical estimate with loss 
\begin{equation}
\label{poulette}
\Norm{(A_{\Phi_j} - \mathcal{I} D^+_{j,K}) x}{s} \leq \frac CK  \Norm{x}{s+2}
\end{equation}
for some constant independent of $h$. Note that in contrast with \eqref{poulette}, we have the better estimates for spectral methods (see \eqref{APK}): for $0\leq s'<s$,
$$
\Norm{(A_\Phi - \mathcal{I} A_{\Phi,K}) x}{s'} \leq \frac{C}{K^{s - s' - r}} \Norm{x}{s}.
$$

\section{Applications}

We now give several original applications of the previous results. The first one revisits the classical error bound for splitting schemes (see \cite{JL}). 

\subsection{Error bounds for splitting schemes}\label{error}

In this section we consider the error of a splitting scheme of order $k\geq2$ for the abstract evolution equation
\begin{equation}\label{edo}
 \dot x=i Ax+i Bx, \quad x \in h^s, 
\end{equation}
where  $A\in\Ac_r$ and $B\in\Ac_\rho$ are Hermitian operators (see section \ref{hermi}), with $r,\rho\in\R$ and  $\rho<r$. 
We also consider space discretization of this problem, of the form
\begin{equation}\label{edoK}
 \dot y =i A^K y+i B^K y \quad u \in \ell^2( \Z_K^d), 
\end{equation}
where $A^\bullet = \{ A^K\}_{K \in 2\N^*}$ and $B^\bullet = \{ B^K\}_{K \in 2\N^*}$ are in $\Ac_r^{\rm per}$ and $\Ac_\rho^{\rm per}$. 

We assume that $iA$ is the generator of a strongly continuous semigroup $e^{itA}$ on $h^0=\ell^2$, and that the flow of $i (A + B)$ and $iB$ are well defined in $h^s$. We assume in particular that for all $s\geq0$, we have the estimates 
\begin{equation}
\label{croisob}
\exists\, M, C >0, \quad 
\forall\, t \in \R, \quad 
\Norm{e^{it ( A + B)}}{\Lc(h^s,h^s)} + \Norm{e^{itA}}{\Lc(h^s,h^s)} + \Norm{e^{it B}}{\Lc(h^s,h^s)} \leq M e^{Ct}  
\end{equation}
where $M$ and $C$ depend on $s$ but not on $t \in \R$. 

In discrete case, and thus in finite dimension, the definition of the flow is not an issue by using matrix exponential. But we assume similarly that the following holds:
\begin{multline}
\label{croisobK}
\exists\, M, C >0, \ \text{ such that }\
\forall\, t \in \R, \quad 
\forall\, K \in 2\N^* \quad \forall\, y \in \ell^2( \Z_K^d) \\
\Norm{e^{it ( A^K + B^K)}y }{s,K} + \Norm{e^{itA^K} y}{s,K} + \Norm{e^{it B^K} y}{s,K} \leq M e^{Ct}  \Norm{y}{s,K}. 
\end{multline}
We will see below examples where such bounds \eqref{croisob} and \eqref{croisobK} are satisfied. 


Using our new class of discrete pseudo-differential operators we recover and amplify a result of Jahnke-Lubich proved in \cite{JL}:  
\begin{theorem}[local error bounds]\label{thm-local}
Let $\rho,r\in\R$ with $\rho < r$.
\begin{itemize}\item[(i)] Consider $A\in\Ac_r$ and $B\in\Ac_\rho$ and
assume that the bound \eqref{croisob}  holds. Then we have the following local error bounds for the Lie and Strang splitting:  for all $s\geq0$ there exists $C_s>0$ and $\tau_0$ such that for all $x \in \ell^2(\Z^d)$ and $|\tau | \leq \tau_0$, 
\begin{align}\label{error-strang}
&\Norm{ e^{ i\tau A}e^{i\tau B}x -e^{i\tau(A+B)}x}{s}\leq C_s\tau^2\Norm{x}{s+r+\rho-1}\\
\label{error-strang2}
&\Norm{ e^{\frac i2\tau B}e^{i\tau A} e^{\frac i2\tau B}x-e^{i\tau(A+B)}x}{s}\leq C_s\tau^3\Norm{x}{s+2r+\rho-2}.
\end{align}
 \item[(ii)] Consider  $A^\bullet = \{ A^K\}_{K \in 2\N^*}$ in $\Ac_r^{\rm per}$  and $B^\bullet = \{ B^K\}_{K \in 2\N^*}$  in $\Ac_\rho^{\rm per}$ and assume that the bound \eqref{croisobK} holds. Then for all $s\geq0$ there exists $C_s>0$ and $\tau_0$ such that for all $|\tau | \leq \tau_0$, 
 for all $K \in 2 \N^*$ and all $y \in \ell^2(\Z_K^d)$, 
\begin{align}
\label{error-strangK}
&\Norm{e^{ i\tau A^K}e^{i\tau B}y -e^{i\tau(A+B)}y}{s,K}\leq C_s\tau^2\Norm{y}{s+r+\rho-1,K}\\
\label{error-strang2K}
&\Norm{e^{\frac i2\tau B}e^{i\tau A} e^{\frac i2\tau B}y-e^{i\tau(A+B)}y}{s,K}\leq C_s\tau^3\Norm{y}{s+2r+\rho-2,K}.
\end{align}
\end{itemize}
\end{theorem}
Comparing with Theorem 2.1 in \cite{JL}, our result is more general in the following sense:
\begin{itemize}
\item We do not assume that $B$ is bounded.
\item  The crucial assumption in  \cite{JL} is $\Norm{ [A,B]v}{s}\leq c_1 \Norm{(-A)^\alpha v}{s}$ and $\Norm{ [A,[A,B]]v}{s}\leq c_1 \Norm{(-A)^\beta v}{s}$ for some $\alpha,\beta\geq 0$. Here these hypothesis are satisfied ( with $\alpha=\frac{r+\rho-1}{r}$ and $\beta=\frac{2r+\rho-2}{r}$ as soon as $A$ and $B$ belong to $\Ac_r$ and $\Ac_\rho$ respectively. So $\alpha,\beta$ can be negative. Furthermore for the order 2 scheme, Jahnke-Lubich assume $\beta\geq1\geq\alpha$ which is not needed here. It will be important in section \ref{water} to apply our result to water waves model.
\item The local error bounds are readily carried over to discrete estimates by using the properties of $\Ac^{\rm per}$. 
\end{itemize}
\begin{proof}
The proof is essentially the one given in \cite{JL} but with more refined commutator estimates. 
We consider only the continuous case, {\em i.e.} time splitting methods applied to  \eqref{edo}. The discrete case \eqref{edoK} is exactly the same as we will only need commutator estimates \eqref{-1K} and the bound \eqref{croisobK} which are both assumed to be independent of $K$. 
\\
To prove \eqref{error-strang}, owing to \eqref{croisob} and taking $|\tau| \leq \tau_0$, it is equivalent to prove the same bound for the operator
$$
A(\tau) := e^{i \tau (A+B)} e^{- i \tau A} e^{-i \tau B}  - \mathrm{Id}. 
$$
We have $A(0) = 0$ and 
\begin{align*}
\frac{\dd}{\dd \tau}A(\tau)  &= e^{i \tau (A+B)} ( i (A + B) - i A - i  e^{- i \tau A} B e^{i \tau A} )  e^{- i \tau A}e^{-i \tau B} \\
&=  e^{i \tau (A+B)} ( i B  - i  e^{- i \tau A} B e^{i \tau A} )  e^{- i \tau A}e^{-i \tau B}. 
\end{align*}
Hence using \eqref{croisob}, we have if we assume $|\tau|\leq \tau_0$, 
$$
\Norm{A(\tau) x}{s} \lesssim_{s,\tau_0} \int_0^\tau \Norm{(B -  e^{- i \sigma A} B e^{i \sigma A}) x(\sigma)}{s} \dd \sigma
$$
where $x(\sigma) = e^{- i \tau A}e^{-i \tau B} x$ satisfies $\Norm{x(\sigma)}{s} \lesssim_{s,\tau_0} \Norm{x}{s}$ for all $s$. 
Let $B(\sigma)= B -  e^{- i \sigma A} B e^{i \sigma A}$. We have $B(0) = 0$ and $\frac{\dd}{\dd \sigma}B(\sigma) = -i  e^{- i \sigma A} [A,B] e^{i \sigma A}$. 
Hence for a given $\tilde x$ we have 
$$
\Norm{(B -  e^{- i \sigma A} B e^{i \sigma A}) \tilde x}{s} \lesssim_{s,\tau_0} \int_0^\sigma \Norm{[A,B] e^{i \alpha A} \tilde x}{s} \dd \alpha. 
$$
Thus, using that  $[A,B] \in \Ac_{r + \rho -1}$ (see Proposition \ref{Prop-Struct}) and Lemma \ref{op}, we deduce 
$$
\Norm{A(\tau) x}{s} \lesssim_{s,\tau_0} \int_0^\tau \int_{0}^\sigma  \Norm{[A,B] e^{i \alpha A} x(\sigma) }{s}\dd \sigma \dd \alpha  \lesssim  \tau^2\Norm{x}{s + r + \rho - 1}. 
$$
which yields the result. \\
To prove \eqref{error-strang2}, we proceed similarly by considering 
$$
A(\tau) := e^{i \tau (A+B)} e^{- i \frac{1}{2} \tau A} e^{-i \tau B} e^{- i \frac{1}{2} \tau A} - \mathrm{Id}. 
$$
We have 
$$
\frac{\dd}{\dd \tau}A(\tau) =  e^{i \tau (A+B)}B(\tau)  e^{- i \frac{1}{2} \tau A} e^{-i \tau B} e^{- i \frac{1}{2} \tau A} 
$$
with 
\begin{align*}
B(\tau) &= i (A + B) - i \frac12 A - i e^{- i \frac{1}{2} \tau A} B  e^{i \frac{1}{2} \tau A}-  \frac12 i e^{- i \frac{1}{2} \tau A} e^{-i \tau B} A e^{i \tau B} e^{i \frac{1}{2} \tau A}\\
&= + i(   B -     e^{- i \frac{1}{2} \tau A} B  e^{i \frac{1}{2} \tau A}) +   \frac{i}{2}  e^{- i \frac{1}{2} \tau A}  ( A - e^{-i \tau B} A e^{i \tau B}) e^{i \frac{1}{2} \tau A}
\\ & =  i e^{- i \frac{1}{2} \tau A} C(\tau) e^{ i \frac{1}{2} \tau A}
\end{align*}
where 
$$
C(\tau) =   e^{i \frac{1}{2} \tau A} B e^{- i \frac{1}{2} \tau A} -      B  + \frac12 A - \frac12 e^{-i \tau B} A e^{i \tau B}. 
$$
We have $C(0) = B(0) =  0$. Moreover, we have 
$$
\frac{\dd^k}{\dd \tau^k} C(\tau) =  \big(\frac{i}{2}\big)^k e^{i \frac{1}{2} \tau A} \mathrm{ad}_A^k B e^{-i \frac{1}{2} \tau A} 
- (-i)^k \frac12 e^{-i \tau B} (\ad^k_B  A)  e^{i \tau B}
$$
where $\mathrm{ad}_BA = [B,A]$. Note that using Proposition \ref{Prop-Struct} assertion {\em (ii)}, we have that $\ad^k_B  A \in \Ac_{r  + k\rho -k}$ and 
$\ad^k_A  B \in \Ac_{\rho + kr - k}$. As $\rho < r$, we deduce 
using \eqref{croisob} that for $|\tau |\leq\tau_0$, we have 
$$
\Norm{\partial_\tau^{k} C(\tau) x}{s} \lesssim_{s,k,\tau_0} \Norm{x}{s + k\rho -k}. 
$$
This leads to the result by using $k = 2$ and the fact that $C(0) = C'(0) = 0$. 
\end{proof}

\begin{remark}
\label{remhigh}
For a number $k \geq 1$, we denote by $SP_k(\tau,A,B)$ a splitting method of order $k$ for \eqref{edo} with time step $\tau$. High order splitting methods can be easily constructed by using composition algorithms, see for instance \cite{BCM08,HLW06} for a review. 
By using more elaborated algebraic formalism coming from the Baker-Campbell-Hausdorff formula, we infer the following: considering a splitting scheme of order $k\geq2$ applied to \eqref{edo} and \eqref{edoK}, then for all $s\geq0$ there exists $C_{s,k}>0$ and $\tau_0$ such that for all $x \in \ell^2(\Z^d)$ and $|\tau | \leq \tau_0$, $y \in \ell^2(\Z_K^d)$ and $K \in 2 \N^*$, 
\begin{align}\label{error-k}
&\Norm{ SP_k(\tau,A,B)x-e^{i\tau(A+B)} x}{s}\leq C_{s,k}\tau^{k+1}\Norm{x}{s+kr+\rho-k}\\
\label{error-k3}
&\Norm{ SP_k(\tau,A^K,B^K) y -e^{i\tau(A+B)} y}{s,K}\leq C_{s,k}\tau^{k+1}\Norm{y}{s+kr+\rho-k,K}
\end{align}
A complete proof in the general case is however out of the scope of this paper.
\end{remark}
\begin{remark}\label{rk2}
 It can be observed in the estimated \eqref{error-k} that the derivative loss, $\rho+k(r-1)$, decreases with the order of the scheme as soon as the order of $A$ and $B$ is strictly smaller than 1. Hence if furthermore $r<1$, there exists $k\geq2$ such that the error bound in the splitting scheme of order $k$ does not require any loss of derivative, i.e. for any $s\geq0$ there exists $C_s>0$ such that:
\begin{equation}\label{error-k2}
\| SP_k(\tau,A,B)x-e^{i\tau(A+B)}v\|_s\leq C_s\tau^{k+1}\|x\|_{s}.
\end{equation}
We will give below an example of such situation for the Strang splitting $(k =2)$ in the case of the water wave system. 
\end{remark}

\begin{remark}
We provide here only local error estimates, but global estimates can be easily obtained by following the classical argument of \cite{JL}. 
\end{remark}
 
\begin{remark}
\label{remsymp}
In the case of symplectic system of the form \eqref{eqsymp}, the previous theorem readily applies, as the commutator of block matrix expresses in terms of commutators of the sub-matrices. Hence the previous theorem hold true for symplectic systems 
\begin{equation}
\label{edosymp}
\dot x = (\boldsymbol{S}_1 + \boldsymbol{S}_2) x
\end{equation}
where $x \in h^s \times h^s$, and $\boldsymbol{S}_1$ and $\boldsymbol{S}_2$ satisfy the decomposition \eqref{eqsymp}, with blocks that are of order less than $r$ and $\rho$ respectively. 
\end{remark}

\subsection{Growth of Sobolev norms\label{sobgrr}}

We give now an example of situations where the commutator estimates can be used to prove long time estimates of Sobolev norm infering in particular assumptions \eqref{croisob} and \eqref{croisobK}. We consider again the equations \eqref{edo} and \eqref{edoK}, but we assume that $A$ and $A^K$ are diagonal operators. In this situation, we can control the evolution of the Sobolev norm of the solutions as follows: 
\begin{theorem} Let $\rho < 1$ and $r > \rho$. Let $A = A_\Phi \in \Ac_r$ where $\Phi$ satisfies the assumption {\em (i)} of Lemma \ref{phietv}, and for all $K \in 2 \N^*$, let $A^K = A_{\Phi,K}$ the spectral approximation of $A_\Phi$ as defined in Lemma \ref{phietvK}. Let $t \mapsto B(t)$ a continuous application from $\R$ to $\Ac_\rho$ with $B(t)^* = B(t)$, and let  $t \mapsto B^\bullet(t) = \{B^K(t)\}_{K \in 2 \N^*}$ a continuous application from $\R$ to $\Ac_\rho^{\rm per}$ be such that the $B^K(t)$ are hermitian for all $K$ and all $t$.  Assume that for all $\alpha,N$, there exists $C_{\alpha,N,\rho}$ such that 
\begin{equation}
\label{assB}
\forall\, t\in \R \quad \quad \Norm{B(t)}{\alpha,N,\rho}+  \DNorm{B^\bullet(t)}{\alpha,N,\rho} \leq C_{\alpha,N,\rho}. 
\end{equation}
Let $x(t) \in h^s$ and $x^K(t) \in \ell^2(\Z_K^d)$ be the solution of the systems
$$
 \dot x=i Ax+i B(t)x, \quad \mbox{and} \quad \dot x^K =i A^K x^K+i B^K(t) x^K. 
$$
Then we have for all $s\geq 0$, the existence of a constant $C_{s,\rho}$ such that with the notation $\langle t \rangle = (1 + |t|^2)^{\frac12}$, 
$$
\left|
\begin{array}{ll}
\forall \, t \in \R\quad  & \Norm{x(t)}{s} \leq C_{s,\rho} \langle t \rangle^{\frac{s}{1 - \rho}} \Norm{x(0)}{s}, \\[1ex]
 \forall\, K \in 2\N^*\quad \forall \, t \in \R\quad  &\Norm{x^K(t)}{s} \leq C_{s,\rho} \langle t \rangle^{\frac{s}{1 - \rho}} \Norm{x^K(0)}{s}. 
\end{array}
\right.
$$
\end{theorem}
\begin{proof}
We consider again only the continuous case, the discrete case being readily obtained by using the same calculations. Note that the proof follows arguments that can be found in \cite{BFG20}.

 Let $(x,y) = \sum_{n \in \Z^d} \bar x_n y_n$ be the standard $\ell^2$ scalar product. First we note that as $A$ and $B(t)$ are hermitian, the $\ell^2$ norm of $x$ is preserved for all times: we have $\Norm{x(t)}{0} = \Norm{x(0)}{0} = (x(0),x(0))^\frac12$.  Let $D$ be the diagonal operator defined by $D(m,n) = (1 + |n|^2)^{\frac12} \delta_{mn}$ for $(m,n) \in \Z^d \times \Z^d$. We have of course $\Norm{x}{s}^2 = (D^s x , D^s x) \in \R$. We calculate then that 
\begin{align*}
\frac{\dd}{\dd t} \Norm{x}{s}^2  & = \mathrm{Re} (D^s x, i D^s (A + B(t) )x ) \\
& = - \mathrm{Im} (D^s x,(A + B(t) )D^s x )  - \mathrm{Im} (D^s x, [D^s , A + B(t) ]x ) 
\end{align*}
As $A$ and $B(t)$ are Hermitian, the first term vanishes, and as $A$ is diagonal, it commutes with $D^s$. Hence, we have 
$$
\left|\frac{\dd}{\dd t} \Norm{x}{s}^2\right| \leq \Norm{x}{s} \Norm{[D^s,B(t)]x }{\ell^2} \lesssim \Norm{x}{s}  \Norm{x}{{s + \rho -1}} 
$$
where the last bound is obtained using commutator estimates and the assumption \eqref{assB}. By using a comparison Lemma with the ordinary differential equation $\dot y = C \sqrt{y}$ (see for instance Lemma 5.2 in \cite{BFG20}) we obtain 
$$
\Norm{x(t)}{s} \leq \Norm{x(0)}{s} + C \int_0^t \Norm{x(\sigma)}{s + \rho -1}. 
$$
for some constant $C$ independent of $t$. By using the preservation of the $\ell^2$ norm, we obtain (take $s=1-\rho$ in the previous estimate)
$$
\Norm{x(t)}{1- \rho} \lesssim \Norm{x(0)}{1-\rho} +  |t| \Norm{x(0)}{0}
$$
and by induction, for all $k \in \N$, 
$$
\Norm{x(t)}{k(1- \rho)} \lesssim \langle t \rangle^{k} \Norm{x(0)}{k(1 - \rho)}
$$
which shows the result by interpolation. 
\end{proof}
\begin{remark}
The previous result in the discrete case with a uniform bound in term of $K$ is original. Note however that it is only interesting form times smaller than $t < K^{1- \rho}$ as we always know that $\Norm{x(t)}{s,K} \leq K^{s} \Norm{x(t)}{0,K}=K^{s} \Norm{x(0)}{0,K}$.
We refer to \cite{FJ15} for the analysis of growth of Sobolev norm for fully discrete splitting schemes. 
\end{remark}

\subsection{Convergence without loss for  water wave models}\label{water}
As an interesting and non trivial example of application, 
we consider a water-wave models with non constant topography. In this section $x \in \T$, $\zeta: \T \to \R$ models the free surface elevation and $\psi: \T \to \R$ models the trace of the velocity potential at the surface.   The function $b:\T \to \R$ reflects the effect of the topography. The linearized water wave model around the flat surface and in presence of topography reads 
$$
\left|
\begin{array}{l}
\partial_t \zeta - G^N[b] \psi= 0   \\[1ex] 
\partial_{t} \psi +  \zeta = 0 
\end{array}
\right.
$$
where $G^N[b]$ is the Dirichlet to Neumann operator, see \cite{Lan13} and the reference therein. We can expand it with respect to $b$. At first order we can expand the operator in powers of $b$ as 
$
G^N[b]  = \Omega^2  + L_1(b) + O(b^2).
$
If we retain only the first term in powers of $b$ we obtain a system of the form 
\begin{equation}
\label{WW1}
\left|
\begin{array}{l}
\partial_t \zeta - \Omega^2 \psi= G\nabla \cdot ( b(x) \nabla G \psi) , \\[1ex] 
\partial_{t} \psi +  \zeta = 0,  
\end{array}
\right.
\end{equation}
where several choices for the operator $G$ can be made. We can retain in a general approximation (see \cite[Section 3.7.2]{Lan13} and \cite{CLS12})
\begin{equation}
\label{Om}
\Omega^2 = \frac{1}{\sqrt{\mu}}|D| \tanh( \sqrt{\mu} |D|) , \quad G = \mathrm{sech}( \sqrt{\mu}|D|). 
\end{equation}
where $\mu$ is a small parameter, 
and
$|D| = |-i \nabla_x|$ is the Fourier multiplier by $|k|$ for $k\in\Z$.  
The limit $\mu \to 0$ (and thus $\Omega^2 = |D|^2$) yields the linearized St-Venant equations with
$  G = 1.
$ 
Several other models (when $G$ is not trivial) with rational approximations of the water wave operators can be derived (for instance the Boussinesq approximations \cite[Section 5.1.3]{Lan13}). 

Note moreover that splitting methods are particularly adapted to \eqref{WW1}: the system can be divided at least into to systems 
\begin{equation}
\label{splitWW}
\left|
\begin{array}{l}
\partial_t \zeta - \Omega^2 \psi=0 \\[1ex] 
\partial_{t} \psi +  \zeta = 0
\end{array}
\right.
\quad \mbox{and}\quad
 \left|
\begin{array}{l}
\partial_t \zeta = G\nabla \cdot ( b(x) \nabla G \psi)  \\[1ex] 
\partial_{t} \psi = 0
\end{array}
\right.
\end{equation}
which can both be solved explicitly: the first one is diagonal in Fourier, and the second enjoys the conservation $\psi(t) = \psi(0)$ which allows an explicit solution $\zeta(t) = \zeta(0) + t G\nabla \cdot ( b(x) \nabla G \psi(0))$ which can be computed easily using Fast Fourier transformations. 
Note that the first system could be also split into two pieces, yielding to Verlet-like algorithms, while the explicit solution corresponds to a Deuflhard algorithm, in the usual terminology of highly oscillatory systems (see \cite[Chapter XIII]{HLW06}). 

To analyze the convergence of splitting schemes based on these decomposition, we make the usual change of variable for wave like systems: We define the new (symplectic) variables
\begin{equation}
\label{sympch}
\begin{pmatrix}
\xi \\ v
\end{pmatrix} = \begin{pmatrix} 
\Omega^{-\frac12}& 0 \\ 0 & \Omega^{\frac12}
\end{pmatrix}\begin{pmatrix}
 \zeta \\ \psi
\end{pmatrix}.
\end{equation}
After calculations, we find that the  Hamiltonian system \eqref{WW1} writes in the new variables
\begin{align}
\nonumber
\frac{\dd}{\dd t}\begin{pmatrix}
\xi \\ v
\end{pmatrix} 
&= \begin{pmatrix}
0 &\Omega   \\
-\Omega & 0
\end{pmatrix}
\begin{pmatrix}
\xi \\ v \end{pmatrix} + \begin{pmatrix}
0 & G \Omega^{-\frac12}\nabla \cdot ( b(x)G \Omega^{-\frac12}  \nabla \\
0 & 0
\end{pmatrix}
\begin{pmatrix}
\xi \\ v \end{pmatrix} \\
&:= \boldsymbol{S}_1
\begin{pmatrix}
\xi \\ v \end{pmatrix} + \boldsymbol{S}_2\begin{pmatrix}
\xi \\ v \end{pmatrix}
\label{wat}
\end{align} 
by using the notation \eqref{defS} for symplectic operators. 
The energy associated with the system is 
$$
H(\xi,v) = \frac12 \int \Omega |\xi|^2 + \Omega |v|^2 + b(x)|G \Omega^{-\frac12}  \nabla v|^2.
$$
Moreover, the flows $e^{t \boldsymbol{S}_1}$ and $e^{t \boldsymbol{S}_2}$ can be easily implemented: the first one decouples in Fourier modes, and the second one is triangular and can be calculated explicitly as explained above. 

In the water wave case \eqref{Om}, we have by using Lemma \ref{phietv} that $\Omega \in \Ac_{\frac12}$. Moreover, we  show that the operator 
$$
A = G \Omega^{-\frac12}\nabla \cdot b(x)G \Omega^{-\frac12}  \nabla 
$$
is smoothing in the following sense: $A \in \Ac_{\rho}$ for all $\rho < 0$. Indeed, in dimension 1  
this operator corresponds to the Fourier matrix with coefficients 
\begin{equation}
\label{Anm}
A(n,m) = G_n \Omega_{n}^{-\frac12} \hat b_{n-m}  G_m \Omega_{m}^{-\frac12}  (in) (im). 
\end{equation}
Asymptotically, for large $n$ and $m$, we have 
$$
| A(n,m)| \lesssim e^{- |n| - |m|} \hat b_{n-m}. 
$$
Thus we see that, even for rough bottom $b$ (with bounded Fourier coefficients for instance), $A \in \Ac_{\rho}$ for all $\rho \leq 0$. 
Hence in this situation, applying Theorem \ref{thm-local}, splitting methods based on the decomposition \eqref{splitWW} converge in any Sobolev space without loss. 
We thus summarize the results of Theorem \ref{thm-local} applied to this situation\footnote{It is also easy to prove the stability estimates \eqref{croisob}}: 
\begin{theorem}
With the notation \eqref{wat}, if $b \in L^\infty(\T)$, we have the following local error estimates for the water wave equation with  \eqref{Om} then for all $s\geq0$ there exists $C_s>0$ and $\tau_0$ such that for all $x  = (\xi,v)^T \in h^s \times h^s$ and $|\tau | \leq \tau_0$, 
\begin{align*}
&\Norm{ e^{ i\tau \boldsymbol{S}_1}e^{i\tau \boldsymbol{S}_2}x -e^{i\tau(\boldsymbol{S}_1 + \boldsymbol{S}_2)}x}{s}\leq C_s\tau^2\Norm{x}{s},\\
&\Norm{ e^{\frac i2\tau \boldsymbol{S}_1}e^{i\tau \boldsymbol{S}_2} e^{\frac i2\tau \boldsymbol{S}_1}x-e^{i\tau(\boldsymbol{S}_1 + \boldsymbol{S}_2)}x}{s}\leq C_s\tau^3\Norm{x}{s}.
\end{align*} 
%
\end{theorem}
Note that these results translate automatically to the original variables by using the change of variables \eqref{sympch}. Theorem \ref{thm-local} can also be applied in various situations where $\Omega$ and $G$ are of some given orders. 

\subsection{Normal form as preconditioners}\label{precon}

Considering splitting schemes for \eqref{edo}, we have seen in Remark \ref{rk2} that when  $A$ is of order $r <1$, it is  possible to find a splitting methods based on the underlying decomposition that converge without loss of derivative. In this section, we show that in many situations, it is possible to make a change of variable putting the system into this form. It is based on a normal form transformation in the spirit of \cite{BGMR}. 

Rather than giving too general result, we will focus on the following system: We consider the Schr\"odinger equation in dimension $1$, 
\begin{equation}
\label{LS0}
 \partial_t u =-i \Delta u + iV(x) u
\end{equation}
where $V$ is a smooth potential. 
Writing $x = \hat u$, the equation becomes
\begin{equation}
\label{LS}
\partial_t  x = A x + Bx \quad x = (x_a)_{a\in \Z} = (\hat u_a)_{a \in \Z},
\end{equation}
where $A$ and $B$ are the matrix defined by 
$$
A(m,n) = |m|^2 \delta_{m,n}\quad \mbox{and} \quad B(m,n) = \hat V(m-n) = (\Fc V)(m-n).  
$$
So in this case $A\in\Ac_2$ and $B\in\Ac_0$ and if we apply directly Theorem \ref{thm-local}, we will obtain a Lie spiting scheme that converges with a loss of 1 derivative and a Strang spliting scheme that converges with a loss of 2 derivatives.\\
Now let us define the operator $X$ by the formula
$$
X(m,n) = 
\left\{
\begin{array}{l} - \displaystyle \frac{\hat V(m-n)}{i (|m|^2 - |n|^2)}  \quad \mbox{when} \quad m \neq \pm n \\[2ex]
0 \quad \mbox{for} \quad m = \pm n
\end{array}
\right.
$$
By using the relation $m^2 - n^2 = (m + n) (m -n)$, we deduce that $|m^2 - n^2| \geq |m| + |n| $ for $m \neq \pm n$, and we deduce easily that $X \in \Ac_{-1}$. Furthermore $X$ is hermitian.\\
In particular, $X$ is a bounded operator from $h^s$ to $h^s$, and 
we can define the transformation $y = e^{iX} x$, from $h^s$ to $h^s$ for all $s\geq0$. 
 
In the new variable $y = e^{i X} x$ the system  reads
$$
\dot y = i e^{iX} (A + B) e^{-iX} y
$$
with $A \in \Ac_2$ and $X \in \Ac_{-1}$. Let ${\rm ad}_X(A) = i [X,A]$, we have 
$$
{\rm ad}_X^j(A) \in \Ac_{2 - 2j}.
$$
In particular, for $j \geq 2$, we have ${\rm ad}_X^j(A) \in \Ac_{-2}$.\\
Now we taylor expand 
$$
e^{i  X}\, (A + B) \,  e^{-i X } = A + B + i [X,A] + i [X,B] + R
$$
where the remainder
$$
R := \int_0^1(1-s)^{2} \e^{-i s  X}\, {\rm ad}_X^{2}(A + B)\, \e^{i sX} \dd s \in \Lc(h^{s},h^{s+2})
$$
is smoothing and gain $2$ derivatives. 
Moreover, as $B \in \Ac_0$ and $X \in \Ac_{-1}$, we have $[X,B] \in \Ac_{-2}$ which is also smoothing and gain $2$ derivatives. 

Eventually, we have by definition of $X$ 
$$
A + B + i [X,A] = A + Z
$$
where 
$$
Z(m,n) = B(m,n) \quad \mbox{for} \quad m = \pm n, \quad Z(m,n) = 0 \quad \mbox{for} \quad m \neq \pm n.
$$
In particular $Z$ is block diagonal and can be easily implemented. 

So far, we have shown that the equation in $y$ can be written 
$$
\dot y = i(A + Z +R)y
$$
where $R \in \Lc(h^{s},h^{s+2})$. 

We can then split the system into 
$$
\dot y = i(A + Z) y \quad \mbox{and} \quad \dot y = iRy. 
$$
The first system is block diagonal and can be easily implemented, and the $R$ part is smoothing and is thus nonstiff. We can therefore easily implement $e^{i (A + Z)}$ and $e^{i R}$ and use splitting approximations. 

In this case, the main error is dominated by the commutator 
$$
[A + Z,R] \in \Lc(h^s,h^s). 
$$
which can be seen by noticing that $R \in \Lc(h^s, h^{s+2})$ and $A + Z \in \Lc(h^s,h^{s-2})$ for all $s$. 
Therefore we obtain the following result: 

\begin{proposition}\label{yesmyfriend}
The Lie-splitting method $e^{i \tau (A + Z)} e^{i \tau R}$ to approximate the solution of $e^{i \tau (A + Z + R)}$ converges without loss of derivative. As a consequence, the scheme
$$
e^{-i X}e^{i \tau (A + Z)} e^{i \tau R} e^{iX}
$$
defines an order $1$ scheme for the equation \eqref{LS} without loss of derivative:  we have 
$$
\Norm{e^{i \tau(A + B) }x - e^{-i X}e^{i \tau (A + Z)} e^{i \tau R} e^{iX} x}{s} \leq C \tau^2 \Norm{x}{s}
$$
for $\tau$ small enough, $s \geq 0$ and a constant $C$ independent of $x$. 
\end{proposition}

From the implementation point of view,  $e^{iX}$ can be seen as a preconditioner, that only need to be evaluated at the beginning and end of the simulation, owing to 
$$
\Big( e^{-i X}e^{i \tau (A + Z)} e^{i \tau R} e^{iX} \Big)^n = e^{-i X} \Big( e^{i \tau (A + Z)} e^{i \tau R} \Big)^n e^{iX}.
$$
Furthermore we stress out that, since $Z$ is block diagonal and $R$  is smoothing , the splitting scheme $e^{i \tau (A + Z)} e^{i \tau R} $ can be easily implemented.\\
We conclude by emphasizing that this normal form strategy, viewed as a preconditioner construction, can be extended to any order of approximation and can be generalized to many situations, see \cite{BGMR}. Note also that, again, this result can be  translated {\em mutatis mutandis} to pseudo-spectral and finite difference approximations using section \ref{per}. A general study is however out of the scope of this paper. 

\begin{appendix}
\section{Young inequality for convolution}
Let $x = (x_n)_{n \in \Z^d}$ and $y = (y_n)_{n \in \Z^d}$ two sequences. We define $z = x * y = (z_n)_{n \in \Z^d}$ the sequence 
$$
z_n = \sum_{\substack{ p,q \in \Z^d \\ n = p + q}} x_p y_q, \quad n \in \Z^d. 
$$
We also define for $p \geq 1$, 
$$
\Norm{x}{\ell^p} = \left( \sum_{n \in \Z^d} |x_n|^p \right)^{\frac{1}{p}}. 
$$
And we recall the following H\"older inequality: for two sequences $x, y$
$$
\left|\sum_{k} x_k y_k \right| \leq \Norm{x}{\ell^p} \Norm{y}{\ell^q} \quad \mbox{for} \quad 1 = \frac{1}p + \frac1q, 
$$
which is itself a consequence of the Young inequality for product: $\forall a,b\geq0$ we have  $ab \leq \frac{a^p}{p} + \frac{b^q}{q}$. This H\"older inequality is easily generalized by induction to 
$$
\left|\sum_{k} \prod_{i = 1}^N x_k^{(i)}  \right| \leq \prod_{i = 1}^N \Norm{x^{(i)}}{\ell^{p_i}} \quad \mbox{for} \quad \sum_{i = 1}^N \frac{1}{p_i} = N, 
$$
for any sequences $x^{(i)}$, $i = 1, \ldots, N$ with $N \in \N$.  
\begin{lemma}
\label{lemyoung}
For two sequences $x$ and $y$ indexed by $\Z^d$, we have 
$$
\Norm{x*y}{\ell^r} \leq \Norm{x}{\ell^p} \Norm{y}{\ell^q} , \quad \mbox{for} \quad 1 + \frac{1}{r} = \frac{1}{p} + \frac{1}{q}. 
$$
\end{lemma}
\begin{proof}
Let us denote $z = x * y$,  we  have
\begin{align*}
|z_n| &\leq \sum_{k \in \Z^d} |x_{n-k}| |y_k| \\
&= \sum_{k \in \Z^d} (|x_{n-k}|^p |y_k|^q)^{\frac{1}{r}}   | x_{n-k}|^{\frac{r- p}{r}} |y_k|^{\frac{r-q}{r}} \\
&\leq \left( \sum_{k} |x_{n-k}|^p |y_k|^q \right)^{\frac{1}{r}} \left( \sum_k | x_{n-k}|^p \right)^{\frac{r -p}{rp}}  \left( \sum_k |y_k|^q \right)^{\frac{r - q}{rq}}\\
&=  \left( \sum_{k} |x_{n-k}|^p |y_k|^q \right)^{\frac{1}{r}}  \Norm{x}{\ell^p}^{\frac{r-p}{r}}  \Norm{y}{\ell^q}^{\frac{r-q}{r}}, 
\end{align*}
where we used the generalized trilinear H\"older inequality for the decomposition 
$$
\frac{1}{r} + \frac{r -p}{rp} + \frac{r - q}{rq} = \frac{1}{p} + \frac{1}{q} - \frac{1}{r} = 1. 
$$
Then we obtain 
\begin{align*}
\sum_{n} |z_n|^r &\leq \Norm{x}{\ell^p}^{r-p}  \Norm{y}{\ell^q}^{r-q}\left( \sum_{k,n} |x_{n-k}|^p |y_k|^q \right) = \Norm{x}{\ell^p}^{r}  \Norm{y}{\ell^q}^{r}. 
\end{align*}
\end{proof}

\end{appendix}

\end{document}